\documentclass[11pt]{amsart}

\usepackage[utf8]{inputenc}
\usepackage{amsmath,amssymb,amsthm,bbm,centernot,xcolor,tikz,upgreek,esint,float,comment,multicol,wrapfig,mathrsfs}
\usepackage{hyperref}
\usepackage[shortlabels]{enumitem}
\usepackage[letterpaper,margin=1.1in]{geometry}
\allowdisplaybreaks

\newcommand{\g}{\gamma}
\newcommand{\e}{\epsilon}
\renewcommand{\t}{\tau}

\newcommand{\dee}{\partial}
\newcommand{\ove}{\overline}
\newcommand{\eps}{\varepsilon}
\renewcommand{\gamma}{\upgamma}
\newcommand{\N}{\mathbb{N}}

\newcommand{\R}{\mathbb{R}}

\newcommand{\cal}{\mathcal}

\newcommand{\bw}{\textbf{w}}
\newcommand{\bW}{\textbf{W}}
\renewcommand{\S}{\mathbb{S}}

\newcommand{\W}{\mathcal{W}}

\renewcommand{\d}{\mathrm{d}}

\newcommand{\til}{\widetilde}

\DeclareMathOperator\length{length}

\let\Im\relax
\DeclareMathOperator\Im{Im}

\DeclareMathOperator\diam{diam}
\DeclareMathOperator\dist{dist}
\DeclareMathOperator\card{card}

\DeclareMathOperator\Mod{Mod}

\let\emptyset\varnothing

\newcommand{\lsim}{\lesssim}

\newtheorem{Thm}{Theorem}[section]
\newtheorem*{Thm*}{Theorem}
\newtheorem{Prop}[Thm]{Proposition}
\newtheorem*{Prop*}{Proposition}

\newtheorem{Lem}[Thm]{Lemma}
\newtheorem*{Lem*}{Lemma}

\newtheorem*{Cla*}{Claim}
\theoremstyle{remark}
\newtheorem*{Rem*}{Remark}

\theoremstyle{definition}

\newtheorem*{Def*}{Definition}
\newtheorem{quest}[Thm]{Question}
\newtheorem{Rem}[Thm]{Remark}

\numberwithin{equation}{section}

\title{Quasisymmetric rectifiability of uniformly disconnected sets}
\author{Jacob Honeycutt}
\author{Vyron Vellis}

\thanks{J.H. and V.V. were partially supported by NSF DMS grant 2154918.}

\address{Department of Mathematics\\ The University of Tennessee\\ Knoxville, TN 37966}
\email{jhoney12@vols.utk.edu}
\address{Department of Mathematics\\ The University of Tennessee\\ Knoxville, TN 37966}
\email{vvellis@utk.edu}

\date{\today}

\begin{document}

\begin{abstract}
We prove that uniformly disconnected subsets of metric measure spaces with controlled geometry (complete, Ahlfors regular, supporting a Poincar\'e inequality, and a mild topological condition) are contained in a quasisymmetric arc.
This generalizes a result of MacManus \cite{M99} from Euclidean spaces to abstract metric setting.
Along the way, we prove a geometric strengthening of the classical Denjoy-Riesz theorem in metric measure spaces.
Finally, we prove that the complement of a uniformly disconnected set in such a metric space is uniform, quantitatively.
\end{abstract}

\maketitle

\section{Introduction}
In the early 1900s, Denjoy \cite{Denjoy} and Riesz \cite{Riesz} independently proved what has come to be known as the Denjoy-Riesz Theorem, which states that \emph{a compact, totally disconnected subset of the plane is contained in an arc.} Recall that a totally disconnected set is one whose connected components are singletons (or equivalently, one that has topological dimension equal to zero) and an arc is the homeomorphic image of the unit interval. This builds upon work by Zoretti \cite{Zoretti}, which instead used a more general class of arcs than just Jordan arcs. More generally, Moore and Kline \cite{MK1919} came up with a complete characterization of subsets of the plane which are contained in arcs, and this includes the Denjoy-Riesz Theorem as a special case. Later, Zippin \cite{Zippin} provided a complete classification of metric spaces for which the Moore-Kline Theorem holds. A corollary of Zippin's work is that in a compact, connected, locally connected metric space with no local cut-points, every compact totally disconnected set is contained in an arc; see also \cite{Whyburn}. A useful application of the (generalized) Denjoy-Riesz theorem is that for every $n\geq 2$, there is an arc in $\R^n$ that has positive Lebesgue measure, and for every $n\geq 3$ there is a \emph{wild} arc in $\R^n$, that is, an arc that is not ambiently homeomorphic to a line segment.

It is natural to ask whether one can obtain more regularity on the arc which passes through the entire set. David and Semmes \cite{DS91} proved a bi-Lipschitz version of the Moore-Kline Theorem: if $n\geq 3$, then a bounded subset of $\R^n$ is contained in a bi-Lipschitz arc in $\R^n$ if and only if it is bi-Lipschitz homeomorphic to a bounded subset of $\R$. MacManus \cite{MM95} extended the result of David and Semmes to the case $n=2$, which is more difficult since intersecting lines in dimensions greater than two can be modified so that they no longer intersect, but this is not the case in $\R^2$. On one hand, these ``bi-Lipschitz rectifiability'' theorems can be viewed as a topological strengthening of the classical Lipschitz rectifiability theory which has been the focus of the Analyst's Traveling Salesman Theorem \cite{jones,oki}. On the other hand, they can be viewed as rough versions of the classical Whitney extension theorem \cite{W34}.

To generalize the bi-Lipschitz rectifiability results to an abstract metric space $X$, in addition to the conditions of Zippin, it is necessary to assume certain geometric properties which ensure that $X$ has plenty of rectifiable curves. Two such conditions are \emph{Ahlfors regularity} and a \emph{Poincar\'e inequality}. A metric space is Ahlfors $Q$-regular (or simply $Q$-regular) if there exists a measure $\mu$ on $X$ so that the measure of every ball $B(x,r)$ is comparable to $r^Q$. A metric measure space has a Poincar\'e inequality if for every Lipschitz function $u:X \to \R$, the oscillation of $u$ on each ball $B$ is controlled by the average of a ``weak derivative'' of $u$ on $B$. See Section \ref{Sec:Preliminaries} for all relevant definitions. It is known that Ahlfors regular spaces supporting a Poincar\'e inequality must contain quantitatively many rectifiable curves and admit a notion of differentiation \cite{C99}. The two authors along with Zimmerman \cite{HVZ24} showed that in Ahlfors regular metric measure spaces with a suitable Poincar\'e inequality, the bi-Lipschitz rectifiability classification of David and Semmes holds.

The goal of this paper is to provide a \emph{quasisymmetric} version of the Denjoy-Riesz theorem in abstract metric spaces. Roughly speaking, a homeomorphism between metric spaces is quasisymmetric if any three points with comparable mutual distances are mapped to three points with comparable mutual distances in a quantitative fashion. Due to Buerling and Ahlfors \cite{BA56}, quasisymmetric maps are a generalization of conformal maps to the abstract metric setting which have been instrumental in the development of analysis on metric spaces. Quasisymmetric arcs (quasisymmetric images of $[0,1]$) have appeared in geometric function theory, geometric group theory (limits of quasi-Fuchsian groups), complex dynamics (Julia sets of certain rational functions), function spaces (domains of extension for Sobolev and BMO maps), Loewner evolution (traces of driving functions $\lambda$ with H\"older norm less than 4), potential theory (Dirichlet integrals and harmonic bending), and hyperbolic geometry (bi-Lipschitz equivalence of hyperbolic and distance-ratio metrics). See \cite{GehringHag} for a detailed exposition.

To replace ``arc'' by ``quasisymmetric arc'' in the Denjoy-Riesz theorem, necessarily some geometric assumptions should be made about the set $E$. For example, if $E\subset \R^n$ is a Cantor set with positive Lebesgue $n$-measure, then any arc that contains $E$ will have positive Lebesgue $n$-measure so it can not be a quasisymmetric arc \cite[Theorem 4.1]{Vai81}. A quantitative version of total disconnectedness was introduced by David and Semmes \cite{DS97}: a metric space $E$ is \emph{uniformly disconnected} if for all $x\in E$ and all $r \in (0,\diam{E})$ there exists a set $E_{x,r}\subset E$ such that $x\in E_{x,r}$, $\diam{E_{x,r}} \leq r$, and $\dist(E_{x,r},E\setminus E_{x,r})$ is at least a fixed multiple of $r$ (i.e., $E$ has isolated islands at every scale and at every point). Using this notion, MacManus \cite{M99} proved a quasisymmetric version of the Denjoy-Riesz theorem in Euclidean spaces.

\begin{Thm}[{\cite[Theorem 1]{M99}}]\label{thm:MM}
A compact, uniformly disconnected subset of $\R^n$ is contained in a quasisymmetric arc. 
\end{Thm}

Our main focus is to generalize Theorem \ref{thm:MM} to a larger class of metric measure spaces. Towards this end, a topological condition in the spirit of Zippin's no-local-cut-point condition is assumed on the space $X$. We say that a set $E\subset X$ is a \emph{cut-set} if $X\setminus E$ is not connected; otherwise we call $E$ a \emph{full set}. The latter term is inspired by \cite{NS} where a planar set $E$ is called full if $\R^2 \setminus E$ has no bounded components. The assumptions of $Q$-regularity and a $Q$-Poincar\'e inequality on $X$ yield that if $E$ is a compact set such that $X\setminus E$ has no bounded components, then $E$ is a full set; see Lemma \ref{Lem:LargeComponent}. Singletons in such a space $X$ can not be local cut-sets but the same is not true with totally disconnected (or even uniformly disconnected) sets; 
%In particular, there exists a an $n$-regular metric measure space that supports an $n$-Poincar\'e inequality and there exists a uniformly disconnected set $E \subset X$ such that $X\setminus E$ is disconnected; 
see the discussion after Question \ref{quest:QSext}.

%For instance, let $E$ be an $s$-regular self-similar Cantor set in $\R^n$ with $s\in (n-1,n)$. Viewed as a subset of $\S^n$, $E$ is an $s$-regular uniformly disconnected set. Let $X=\S^n\cup_E \S^n$ be the geodesic gluing of two copies of $\S^n$ along $E$. Then $X$ is an $n$-regular metric measure space (with the usual spherical metric and spherical measure) that supports an $(n-1)$-Poincar\'e inequality \cite[Theorem 6.15]{HK98} and $E\subset X$ is a uniformly disconnected set that separates $X$.

The next topological condition is designed to make sure that closed totally disconnected sets do not separate the space. A metric space $X$ is \emph{fully connected} if it is path connected, locally path connected, and every full set is locally full. That is, if $X\setminus E$ is connected and if $U$ is an open, connected set such that $U\cap E \neq \emptyset$ and $E\cap \partial U = \emptyset$, then $U\setminus E$ is connected. For example, all Euclidean spaces $\R^n$ are fully connected, while tori $(\S^1)^n$ are not. Our main theorem is as follows.
%Full connectedness implies that Cantor sets are not ``cut-sets'' in a strong way: if $C$ is a Cantor set and $U$ is an open set that intersects $C$, then $U\setminus C$ is connected.

\begin{Thm}\label{thm:main1}
Let $X$ be a complete, fully connected, unbounded, $Q$-regular metric measure space supporting a $(Q-1)$-Poincar\'e inequality with $Q> 2$. If $E\subset X$ is a compact, uniformly disconnected set, then $E$ is contained in a quasisymmetric arc.
\end{Thm}

This theorem is quantitative in the sense that the quasisymmetric control function depends only on the constants of Ahlfors regularity, the constants of the Poincar\'e inequality, and the constant of uniform disconnectedness. Canonical examples of of metric measure spaces satisfying the assumptions of Theorem \ref{thm:main1} are Carnot groups, which include all Euclidean spaces and all Heisenberg groups. We will not define Carnot groups here but only mention that they are topologically identical to Euclidean space, but geometrically can behave much worse; although they are geodesic spaces, they have much fewer rectifiable curves and surfaces.

The techniques and tools used in the proof of Theorem \ref{thm:MM}, as well as modern proofs of the Denjoy-Riesz theorem, involve several classical tools from Euclidean spaces which are not available here. First, Euclidean spaces have a \emph{dyadic decomposition} of cubes. This decomposition can be used to create a defining sequence for any Cantor set which consists of piecewise linear $n$-manifolds with boundary; see for example \cite[Section 2]{M99} and \cite[Corollary 5.2]{BV19}. While generalizations of dyadic decompositions exist in general metric spaces \cite{HyKa,KaRaSu}, here we take a more simple approach based on open neighborhoods and create ``good'' defining sequences for all compact, totally disconnected sets. Second, in Euclidean spaces, arcs can be approximated in the Hausdorff distance by bi-Lipschitz arcs quantitatively, but this is much less obvious in metric spaces with fractal directions such as the Heisenberg group. We resolve this issue by using the Poincar\'e inequality and Ahlfors regularity in order to generate nice curves; see \textsection\ref{sec:outline} for further discussion.

In fact, the techniques of this paper allow us to establish the following strengthening of the Denjoy-Riesz Theorem. Here, $\dim_H{X}$ denotes the Hausdorff dimension of a metric space $X$.

\begin{Thm}\label{Thm:main2}
Let $X$ be a complete, fully connected, unbounded, $Q$-regular metric measure space supporting a $(Q-1)$-Poincar\'e inequality with $Q> 2$. If $E \subset X$ is a compact, totally disconnected set, then there exists an embedding $f:[0,1] \to X$ such that $E\subset f([0,1])$ and $f$ is locally bi-Lipschitz away from $f^{-1}(E)$. In particular, 
%the Hausdorff dimension 
$\dim_H{f([0,1])} = \max\{1,\dim_{H}{E}\}$.
\end{Thm} 

%Moreover, we believe that full connectedness is unnecessary in Theorem \ref{thm:main1} but our proofs make heavy use of it. The following question, which could be of independent interest, would be crucial in this direction. 
%
%\begin{quest}\label{quest:1}
%Let $X$ be a $Q$-regular metric measure space supporting a $(Q-1)$-Poincar\'e inequality with $Q > 2$, and let $E\subset X$ be uniformly disconnected. Then there exists $p\in (1,Q)$ such that $E$ has zero $p$-capacity. 
%\end{quest}
%
%See \textsection\ref{Sec:ModCap} for the definition of the \emph{capacity}. In general, Question \ref{quest:1} is open even when $X$ is a Euclidean space. It is known that if $X=\R^n$ and $E$ has Hausdorff dimension less than $n-1$, then $E$ has zero $p$-capacity for some $p\in (1,Q)$ \cite{V75}. However, uniformly disconnected sets in $\R^n$ can have Hausdorff dimension arbitrarily close to $n$. 

One can also ask whether a complete classification of quasisymmetrically rectifiable sets exists, in the spirit of Moore-Kline (for arcs) and of David-Semmes and Honeycutt-Vellis-Zimmerman (for bi-Lipschitz arcs). In the bi-Lipschitz classification, the key ingredient is the existence of bi-Lipschitz extension theorems \cite{DS91,HVZ24}. In contrast, quasisymmetric maps are in general much harder to extend and we leave the following as an open question. 

\begin{quest}\label{quest:QSext}
Let $X$ be a complete, fully connected, unbounded, $Q$-regular metric measure space supporting a $(Q-1)$-Poincar\'e inequality for some $Q>2$. If $E\subset \R$ and $f:E \to X$ is a quasisymmetric embedding, is there a quasisymmetric extension $F:\R \to X$?
\end{quest}

Question \ref{quest:QSext} is unknown even in the case where $X$ is a Euclidean space although partial results exist if one increases the dimension of the target of the extension \cite{Vellis}. 

Although we do not know whether full connectedness can be removed from Theorem \ref{thm:main1} and Theorem \ref{Thm:main2}, it can not be removed from Question \ref{quest:QSext}.
%We note that neither of full connectedness, Ahlfors regularity, and Poincar\'e inequality can be removed. For the necessity of the latter two conditions see the discussion immediately after \cite[Theorem 1.2]{HVZ24}. For full connectedness, 
To see that, let $s\in (2,3)$, let $A\subset \R^3$ be an unbounded closed $s$-regular uniformly disconnected set (e.g. a blow-up tangent of an $s$-regular self-similar Cantor set), and let $A' \subset [0,\infty)$ be an unbounded closed Ahlfors regular uniformly disconnected set (e.g. a blow-up tangent of the standard Cantor set at 0). If $X=\R^3\cup_A \R^3$ is the geodesic gluing of two copies of $\R^3$ along $A$ equipped with the Hausdorff 3-measure, then $X$ is 3-regular and it supports a $2$-Poincar\'e inequality \cite[Theorem 6.15]{HK98}. Let now $y_1,y_2 \in X\setminus A$ be two points in different copies of $\R^3$, and let $E = A' \cup \{-2,-1\}$. Then there is a quasisymmetric embedding $f:E \to X$ such that $f(A')=A$, $f(-1)=y_1$ and $f(-2)=y_2$ (a simple variation of \cite[Proposition 15.11]{DS97}) which can not be extended homeomorphically (let alone quasisymmetrically) on $\R$.

We also remark that in Euclidean spaces there exist higher dimensional bi-Lipschitz and quasisymmetric rectifiability results; in particular, if $E \subset \R^n$ is compact and uniformly disconnected, then for any $m\in\{1,\dots,n-1\}$ there exists a quasisymmetric embedding $f$ of $[0,1]^m$ into $\R^n$ that contains $E$ in its image \cite[Theorem 3.8]{BV19}. If, in addition, the Assouad dimension of $E$ is less than $m$, then $f$ can be assumed bi-Lipschitz \cite[Theorem 3.4]{BV19}. While Ahlfors regular metric spaces with Poincar\'e inequalities contain a lot of rectifiable curves, the same is not true with surfaces; for example, the Heisenberg group contains no bi-Lipschitz images of the unit square. It would be very interesting if analogous results were true for certain Carnot groups.

\subsection{Co-uniformity of uniformly disconnected sets}

The notion of local separation of a uniformly disconnected set $E$ naturally lends itself to the idea that there are lots of holes in it; in particular, the complement of $E$ contains a lot of curves. A quantification of this idea was given by MacManus who proved that if $E\subset \R^n$ is uniformly disconnected, then $\R^n \setminus E$ is a uniform set \cite[Theorem 1]{M99}. Recall that a set $U$ is uniform if for any two points in $U$ there is a curve joining them which is not too long compared to the distance of the points and does not get too close to the boundary, aside from possibly near the endpoints. Since their inception by Martio and Sarvas \cite{MartioSarvas}, uniform domains have played a crucial role in geometric function theory, especially in the extendability of Sobolev and BMO functions \cite{GO79,J80,J81}.

The study of uniform domains is prevalent in literature. In particular, uniformity has been explored through concepts like quasihyperbolicity \cite{H87,HVW17,H22}, Gromov hyperbolicity \cite{BHK01,HSX}, uniform disconnectedness \cite{M99}, bounds on topological dimension \cite{V88}, and bounds on Assouad dimension \cite{HVZ24}. We show that MacManus' result generalizes to a large class of metric measure spaces, as given by the following theorem.

\begin{Thm}\label{Thm:Uniformity}
Let $(X,d,\mu)$ be a complete, fully connected, unbounded, $Q$-regular metric measure space supporting a $Q$-Poincar\'e inequality. If $E$ is a uniformly disconnected subset of $X$, then $X \setminus E$ is uniform.
\end{Thm}

The theorem is quantitative in the sense that the constant of uniformity depends only on on the constants of Ahlfors regularity, the constants of the Poincar\'e inequality, and the constant of uniform disconnectedness. Note also that in Theorem \ref{Thm:Uniformity}, we assume a $Q$-Poincar\'e inequality which is weaker than the $(Q-1)$-Poincar\'e inequality which is assumed in Theorem \ref{thm:main1}.

\subsection{Outline of the paper}\label{sec:outline}

In Section \ref{Sec:Preliminaries}, we review some important definitions from analysis on metric spaces and present some preliminary results.

The common ingredient in the proof of all three main theorems is the construction of a ``thickening'' of compact totally disconnected sets in Section \ref{Sec:ADecompositionOfUniformlyDisconnectedSets}. This thickening is a replacement for the cubic meshes used in MacManus' \cite{M99} proof to give a natural decomposition of uniformly disconnected sets in Euclidean space. In particular, Proposition \ref{Prop:DefiningSequence} provides us with a defining sequence $\{D_{n}\}_{n\in\N}$ of open sets in $X$ such that 
\begin{enumerate}
\item for all $n\in\N$, $D_n = D_{n,1}\sqcup \cdots\sqcup D_{n,k_n}$ where $D_{n,1},\dots,D_{n,k_n}$ are mutually disjoint full domains with diameters at most a fixed multiple of $2^{-n}$,
\item for all $n\in\N$, each $D_{n,i}$ contains some $D_{n+1,j}$, and each $D_{n+1,i}$ is contained in some $D_{n,j}$
%$i\in\{1,\dots,k_{n+1}\}$ there exists $j\in\{1,\dots,k_n\}$ such that $D_{n+1,i}\subset D_{n,j}$,
%\item for all $n\in\N$ and $i\in \{1,\dots,k_n\}$, there exists $j\in\{1,\dots,k_{n+1}\}$ such that $D_{n+1,j}\subset D_{n,i}$,
\item for all $n\in\N$, $\dist(D_{n+1},\partial D_n)>0$, and
\item $E=\bigcap_{n\in\N}\bigcup_{i=1}^{k_n} \overline{D_{n,i}}$.
\end{enumerate}
%$E=\bigcap_{n\in\N}\bigcup_{i=1}^{k_n} D_n$ and , 2) $D_n$ consists of several components each at most a fixed multiple of $2^-n$, and 3) $\dist(D_{n+1},\partial D_n)>0$. 
Moreover, in the case that $E$ is uniformly disconnected, we have quantitative control on the size of $k_n$, diameters of $D_n$, and distances $\dist(D_{n + 1},\dee D_n)$.
%The core idea of the thickening is that it is a way to encapsulate the uniformly disconnected set at different scales by open sets.
Using this decomposition we prove Theorem \ref{Thm:Uniformity} in Section \ref{Sec:Co-uniformity}.

In Section \ref{Sec:Quasisymmetric1}, we set up notation and some preliminary lemmas for the proof of Theorems \ref{thm:main1} and \ref{Thm:main2}. The crux of the proof of Theorem \ref{thm:main1} and Theorem \ref{Thm:main2} is the construction of bi-Lipschitz arcs in each $D_{n,i}\setminus D_{n+1}$ with the first arc going from the boundary of $D_{n,i}$ to the boundary of the first component $D_{n+1,i_1}$ of $D_{n+1}$ inside $D_n$, the second curve going from the boundary of $D_{n+1,i_1}$ to the boundary of the second component $D_{n+1,i_2}$ of $D_{n+1}$ inside $D_n$ and so on, so forth. The last curve, goes from the boundary of the last component of $D_{n+1}$ inside $D_n$ to the boundary of $D_n$; see Figure \ref{fig:curves}. We take extra care so that these curves do not intersect with one another or with curves in other domains $D_{n,j}$. This construction relies heavily on the Poincar\'e inequality and the Ahlfors regularity of the space and is given in Section \ref{Sec:Quasisymmetric2}.

In Section \ref{sec:modif}, we perform modifications around the endpoints of the curves previously defined to ensure that the concatenation is locally bi-Lipschitz, (and globally quasisymmetric when $E$ is uniformly disconnected) and in Section \ref{sec:proof} we give the proofs of Theorem \ref{thm:main1} and Theorem \ref{Thm:main2}. 

%The main tool in the proof is the thickening given by Proposition \ref{Prop:DefiningSequence}.Fixing two points $x$ and $y$ in the complement of a uniformly disconnected set $E$, using properties of the thickening, there are smallest open sets which contain $x$ and $y$ individually, and a smallest open set containing both.We then construct curves which join the boundaries of open sets on one scale to the previous until we reach the common open set, where we concatenate all the curves constructed.

\section{Preliminaries} \label{Sec:Preliminaries}

Here and for the rest of the paper, when we write $C(c_1,\ldots,c_n)$, we mean that the constant $C$ depends only on the constants $c_1,\ldots,c_n$.
When we write $x \lsim_{c_1,\ldots,c_n} y$, we mean there is a constant $C(c_1,\ldots,c_n) > 0$ such that $x \le Cy$.
If the constant $C > 0$ is universal or the constants are understood, then we write $x \lsim y$.

If $X$ is a metric space, $x\in X$, $r>0$, and $E,F\subset X$, then we denote the distance of $x$ to $E$ by $\dist(x,E)$, the distance of $E$ to $F$ by $\dist(E,F)$, the diameter of $E$ by $\diam(E)$, and the open $r$-neighborhood of the set $E$ by
\[ N(E,r) = \bigcup_{y \in E} B(y,r) = \{z \in X : \dist(z,E) < r\}. \]
We also denote the closed $r$-neighborhood of the set $E$ by $\ove{N}(E,r)$.
If $B = B(x,r)$ is a ball in $X$ and $c > 0$, we denote by $cB$ the ball $B(x,cr)$.
Provided $E$ and $F$ are nondegenerate continua (compact and connected sets with more than one point), we denote the \emph{relative distance} of $E$ and $F$ by
\[ \Delta(E,F) = \frac{\dist(E,F)}{\min\{\diam(E),\diam(F)\}}. \]

Given a set $A \subset \R$ and a map $f:A \to X$, we denote by $\Im(f)$ the image of $f$. If $A$ is a closed interval and $f$ is continuous (that is, $f$ is a curve), then we conflate $f$ with its image $\Im(f)$ as a subset of $X$. 

Define the \emph{length} of a curve $\gamma : [a,b] \to X$ by
\[ \length(\gamma) = \sup\left\{\sum_{k = 1}^n d(\gamma(t_k),\gamma(t_{k-1})) : a < t_0 < \cdots < t_n < b\right\}. \]

\subsection{Words}
An \emph{alphabet} is either a set of the form $A=\{1,\dots,n\}$ (in which case we say that we have a finite alphabet) or $A=\N$ (in which case we say that we have an infinite alphabet).
Given an alphabet $A$, we denote a typical element (a word) of $A^k$ by $a_1\cdots a_k$ instead of $(a_1,\ldots,a_k)$.
The number $k$ is the \emph{length} of any word $a \in A^k$, and we write $|a|$ to denote this value.
We use the convention that $A^0 = \{\eps\}$, where $\eps$ denotes the \emph{empty word}, which satisfies $|\eps| = 0$.
We also denote $A^* = \bigcup_{k\geq 0} A^k$, the collection of all words of finite length.
A \emph{dictionary} is any subset of $A^*$.
Lastly, given $w\in A^*\setminus\{\varepsilon\}$, denote by $w^{\uparrow}$ the unique word $u\in A^*$ for which there exists $i\in A$ such that $w = ui$ (sometimes called the \emph{parent} of $w$).

\subsection{Homogeneously totally bounded spaces and Ahlfors regularity} \label{Sec:Regularity}

Inspired by \cite[\textsection2.7]{TuVa}, we say that a metric space $(X,d)$ is $(C,\alpha)$-\emph{homogeneously totally bounded} (abbv. $(C,\alpha)$-HTB) for some $C,\alpha>0$ if for each $x\in X$, and each $0<r<R$, the ball $B(x,R)$ can be covered by at most $C(R/r)^{\alpha}$ many balls of radii $r$. %We say that $X$ is \emph{doubling} if there exist $C,\alpha>0$ such that $X$ is $(C,\alpha)$-HTB. 
Note that if $X$ is $(C,\alpha)$-HTB, then $X$ is $(C,\beta)$-HTB for all $\beta \ge \alpha$. If we do not care for the constant $C$, we say that $X$ is $\alpha$-HTB. The image of a $(C,\alpha)$-HTB space under a bi-Lipschitz map is $(C',\alpha)$-HTB quantitatively.

A metric measure space $(X,d,\mu)$ is called \emph{Ahlfors $(C_0,Q)$-regular} (or simply \emph{$(C_0,Q)$-regular}) for constants $C_0 \ge 1$ and $Q \ge 0$ if for each point $x\in X$ and $r \in (0,\diam{X})$,
\begin{equation}\label{eq:regular}
C_0^{-1}r^Q \le \mu(B(x,r)) \le C_0r^Q.
\end{equation}
If the right inequality in \eqref{eq:regular} holds for all $x\in X$ and $r \in (0,\diam{X})$, then we say that $(X,d,\mu)$ is \emph{upper $(C_0,Q)$-regular}. Note that if $(X,d,\mu)$ is $(C_0,Q)$-regular, then $(X,d,\mathcal{H}^Q)$ is $(C_0',Q)$-regular for some $C_0'$ depending only on $C_0$ \cite[\textsection1.4.3]{MT10}. Here and for the rest of the paper, $\mathcal{H}^Q$ denotes the $Q$-dimensional Hausdorff measure.
It is not hard to see that if $(X,d,\mu)$ is $Q$-regular, then $(X,d)$ is $(C,Q)$-HTB for some $C>0$ depending only on $C_0$ and $Q$.

\subsection{Mappings} \label{Sec:RegMap}

Let $(X,d)$ and $(Y,d')$ be two metric spaces.
Given $L \ge 1$ and $\lambda > 0$, a map $f:X \to Y$ is said to be an \emph{$(L,\lambda)$-quasisimilarity} if for all $x_1,x_2 \in X$,
\[ \frac{\lambda}{L}d(x_1,x_2) \le d'(f(x_1),f(x_2)) \le L\lambda d(x_1,x_2). \]
If $\lambda = 1$, then we say that $f$ is \emph{$L$-bi-Lipschitz}.
If $L = 1$, then $f$ is a \emph{$\lambda$-similarity}.
If $L = \lambda = 1$, then $f$ is an \emph{isometry}.

Given a homeomorphism $\eta:[0,\infty) \to [0,\infty)$, a map $f:X \to Y$ is \emph{$\eta$-quasisymmetric} if for all $x,a,b \in X$ with $x \ne b$,
\begin{equation}\label{eq:QS}
\frac{d'(f(x),f(a))}{d'(f(x),f(b))} \le \eta\left(\frac{d(x,a)}{d(x,b)}\right).
\end{equation}
Since their inception by Beurling and Ahlfors \cite{BA56}, quasisymmetric maps have played an important role in the study of analysis on metric spaces. While bi-Lipschitz maps roughly preserve absolute distances, quasisymmetric maps roughly preserve relative distances. 

It is well-known that under certain conditions on $X$ and $Y$, quasisymmetry is equivalent to a weaker condition, aptly named \emph{weak quasisymmetry}. %Recall that a metric space $X$ is said to be \emph{doubling} if there exists $C\geq1$ such that each ball $B\subset X$ can be covered by at most $C$ many balls of half the radius of $B$.

\begin{Lem}[{\cite[Theorem 10.19]{H01}}]
Let $X$ be a connected, HTB metric space and let $Y$ be a HTB metric space. An embedding $f:X \to Y$ is $\eta$-quasisymmetric if and only if there exists $H\geq 1$ such that for any triple $x,y,z\in X$ with $d(x,y)\leq d(x,z)$ we have $d'(f(x),f(y))\leq H d'(f(x),f(z))$. The function $\eta$, the constant $H$, and the HTB constants of $X$ and $Y$ are quantitatively related.
\end{Lem}

See \cite{H01} for a detailed exposition on the theory of quasisymmetric maps.

\subsection{Porosity and uniform disconnectedness}\label{Subsec:UniformlyDisconnectedSpaces}
We say that a set $E\subset X$ is \emph{$p$-porous} for some $p\geq 1$ if for all $x\in E$ and all $r\in (0,\diam{X})$ there exists a ball $B(y,p^{-1}r) \subset B(x,r)\setminus E$. The following lemma gives a classification of porous sets in Ahlfors regular spaces in terms of their dimension.

\begin{Lem}[{\cite[Lemma 3.12]{BHR01}}]\label{lem:BHR}
Let $(X,d,\mathcal{H}^Q)$ be a $(C_0,Q)$-regular metric space. A set $E\subset X$ is $p$-porous if and only if it is $(C,\alpha)$-HTB for some $C > 0$ and $\alpha \in (0,Q)$. Constants $C$, $Q$, $C_0$, and $\alpha$ are quantitatively related.
\end{Lem}

In light of the discussion in \textsection\ref{Sec:Regularity}, the measure $\mathcal{H}^Q$ in Lemma \ref{lem:BHR} can be replaced by any other $Q$-regular measure. 

We say that a metric space $(X,d)$ is \emph{$c$-uniformly disconnected} for some $c \ge 1$ if for all $x \in X$ and $r \in (0,\diam X)$, there exists a set $X_{x,r} \subset X$ such that $x \in X_{x,r}$, $\diam{X_{x,r}} \le r$, and $\dist(X_{x,r},X \setminus X_{x,r}) \ge c^{-1}r$.

We say that a collection $\{x_0,\dots,x_n\}$ of points in a metric space is a \emph{$\delta$-chain} if $d(x_{i - 1},x_i) \le \delta d(x_0,x_n)$ for each $i \in \{1,\dots,n\}$. We say that the chain \emph{joins} $x_0$ and $x_n$. The following lemma relates the definition of uniform disconnectedness to the nonexistence of $\delta$-chains.

\begin{Lem}[{\cite[Exercise 14.26]{H01}}]\label{Lem:EquivalentDefinition}
A metric space $X$ is $c$-uniformly disconnected for some $c \ge 1$ if and only if there exists $\delta \in (0,1)$ such that for any two points $x,y \in X$, there is no $\delta$-chain joining $x$ and $y$. Here, $\delta$ and $c$ depend only on each other.
\end{Lem}

\subsection{Poincar\'e inequality and modulus of curve families}\label{Sec:Poincare}\label{Sec:ModCap}

Let $(X,d)$ be a metric space. We say that a function $g:X \to [0,\infty]$ is an \emph{upper gradient} of a locally Lipschitz function $u:X \to \R$ if for all points $x$ and $y$ of $X$ and all rectifiable curves $\gamma \subset X$ joining $x$ with $y$,
\[ |u(x) - u(y)| \le \int_{\gamma} g\,\d s. \]

A metric measure space $(X,d,\mu)$ is said to support a \emph{$p$-Poincar\'e inequality} if there exist constants $C,\lambda>1$ such that for every locally Lipschitz function $u:X \to \R$, every upper gradient $g:X \to [0,\infty]$ of $u$, every point $x\in X$, and every $r>0$,
\[ \fint_{B(x,r)} |u - u_{B(x,r)}|\, \d\mu \le C\diam(B(x,\lambda r))\left(\fint_{B(x,\lambda r)} g^p\, \d\mu\right)^{1/p}, \]
where, for an integrable function $f:X \to \R$,
\[ f_{B(x,r)} = \fint_{B(x,r)} f\, \d\mu = \frac{1}{\mu(B(x,r))}\int_{B(x,r)} f\, \d\mu. \]
It follows from H\"older's inequality that if $p < q$ and $(X,d,\mu)$ satisfies a $p$-Poincar\'e inequality, then it also satisfies a $q$-Poincar\'e inequality. Moreover, if $(X,d,\mu)$ is geodesic, then one can take $\lambda = 1$ \cite[Remark 9.1.19]{HKST15}. If we want to emphasize the constants, we say that $(X,d,\mu)$ satisfies an $(p,C)$-Poincar\'e inequality.

%\subsection{Modulus of curve families}\label{Sec:ModCap}

Let $(X,d,\mu)$ be a metric measure space and let $\Gamma$ be a collection of rectifiable curves in $X$.
A function $f:X \to [0,\infty)$ is said to be \emph{admissible} for $\Gamma$ if
\[ \int_{\gamma} f\,\d s \ge 1 \]
for each $\gamma \in \Gamma$. We define the \emph{$p$-modulus} of $\Gamma$ by
\[ \Mod_p(\Gamma) = \inf \left\{ \int_X f^p\,\d\mu : \text{ $f:X \to [0,\infty)$ is admissible for $\Gamma$}\right\}. \]

The following proposition is useful for calculating upper bounds on the modulus of families of curves in upper Ahlfors regular spaces.
\begin{Lem}[{\cite[Lemma 7.18]{H01}}]\label{Lem:AnnulusModulus}
Suppose that $(X,d,\mu)$ is a metric measure space that is upper $(C_0,Q)$-regular. There exists $C_1(C_0,Q)\geq 1$ such that for $x \in X$ and $2r < R$,
\[ \Mod_Q(\Gamma) \le C_1\left(\log\left(\frac{R}{r}\right)\right)^{1 - Q}, \]
where $\Gamma$ is the family of curves joining $\overline B(x,r)$ to $X \setminus B(x,R)$.
\end{Lem}

A metric space $(X,d)$ is said to be \emph{$Q$-Loewner} if there exists a decreasing homeomorphism $\Phi:(0,\infty) \to (0,\infty)$ such that $\Mod_Q(E,F) \ge \Phi(\Delta(E,F))$ for each pair of nondegenerate continua $E$ and $F$ in $X$. Ahlfors regularity and a Poincar\'e inequality imply the Loewner property. 
%(and in fact, Ahlfors regularity and the Loewner property together imply a Poincar\'e inequality).

\begin{Lem}[{\cite[Theorem 5.7]{HK98}}]\label{Lem:Loewner}
Let $(X,d,\mu)$ be a complete, $(C_0,Q)$-regular metric measure space supporting a $(Q,C)$-Poincar\'e inequality. Then $X$ is $Q$-Loewner and the Loewner function $\Phi$ depends only on $C_0$, $Q$, and $C$.
\end{Lem}

%Given a metric measure space $(X,d,\mu)$, a compact set $E\subset$, and a number $p\geq 1$, define the \emph{(Sobolev) $p$-capacity} of $E$ by
%\[ \Capa_p(E) = \inf_{u,g} \int_X (|u|^p + g^p)\, d\mu\]
%where the infimum is taken over all functions $u$ in the Newtonian space $N^{1,p}(X)$ such that $0\leq u \leq 1$ on $X$ and $u=1$ on $E$, and all upper gradients $g$ of $u$; see \cite[(7.2.1)]{HKST15} and \cite[Lemma 7.2.6]{HKST15}. A compact set $E$ has zero $p$-capacity if and only if the $p$-modulus of all non-constant curves intersecting $E$ is zero \cite[Proposition 7.2.8]{HKST15}. 

\subsection{Uniformity} \label{Subsec:Uniformity}

A subset $U$ of a metric space $(X,d)$ is \emph{$c$-uniform} for some $c \ge 1$ if for all $x,y \in \ove{U}$, there exists a curve $\gamma:[0,1] \to \ove{U}$ with $\gamma(0) = x$, $\gamma(1) = y$,
\begin{enumerate}
\item $\length(\gamma) \le cd(x,y)$ and
\item $\dist(\gamma(t),X \setminus U) \ge c^{-1}\dist(\gamma(t),\{x,y\})$ for all $t \in [0,1]$.
\end{enumerate}
If for each pair of points $x,y \in U$, there is a curve $\gamma$ that satisfies condition (1), then we say that $U$ is \emph{$c$-quasiconvex}.

The next two lemmas allow us to assume for the rest of the paper that $X$ is geodesic, since any quasiconvex metric space is bi-Lipschitz equivalent to a geodesic metric space through the identity map.
\begin{Lem}[{\cite[Theorem 17.1]{C99}}] \label{Lem:Quasiconvex}
If $(X,d,\mu)$ is a complete $(C_0,Q)$-regular metric measure space and satisfies a $(p,C)$-Poincar\'e inequality, then there is a constant $c(C_0,Q,p,C) \ge 1$ such that $X$ is $c$-quasiconvex.
\end{Lem}

\begin{Lem}[{\cite[Lemma 8.3.18]{HKST15}}] \label{Lem:Preserve}
Let $d_1$ and $d_2$ be two bi-Lipschitz equivalent metrics on $X$.
If $\mu$ is Ahlfors regular with respect to $d_1$, then it is Ahlfors regular with respect to $d_2$, quantitatively.
Furthermore, if $(X,d_1,\mu)$ supports a $p$-Poincar\'e inequality, then $(X,d_2,\mu)$ supports a $p$-Poincar\'e inequality, quantitatively.
\end{Lem}

\subsection{Bi-Lipschitz approximation of curves and avoiding small sets.}

We end this section with a lemma that allows us---inside a good metric measure space---to approximate any curve by a bi-Lipschitz arc which also avoids a ``small set'' quantitatively.

\begin{Lem}\label{Lem:ApproximatingCurve}
Given $C_0,C,C_1> 1$, $Q>2$, $p\in [1,Q-1)$, $c \in (0,1)$, $\lambda \in (0,c)$, and $\alpha\in (0,Q-p)$, there exist constants $L(C_0,Q,p,C,C_1,\alpha,\lambda)>1$ and $\beta(C_0,Q,p,C,C_1,\alpha,c) \in (0,c)$ with the following property.

Let $(X,d,\mu)$ be a complete, $(C_0,Q)$-regular metric measure space supporting a $(p,C)$-Poincar\'e inequality, let $Y \subset X$ be $(C_1,\alpha)$-HTB, and let $\sigma:[0,1] \to X$ be a curve satisfying $d(\sigma(0),\sigma(1)) \ge c\diam\sigma([0,1])$ and $\dist(\{\sigma(0),\sigma(1)\},Y) \ge c\diam\sigma([0,1])$. There exists an $(L,\diam{\sigma([0,1])})$-quasisimilarity $\gamma:[0,1] \to X$ such that $\gamma(0) = \sigma(0)$, $\gamma(1) = \sigma(1)$, 
%for each $t\in [0,1]$,
\[ \sup_{t\in[0,1]}\dist(\gamma(t),\sigma([0,1])) \le 4\lambda\diam{\sigma([0,1])}, \quad\text{and}\quad \inf_{t\in[0,1]}\dist(\gamma(t),Y) \ge \beta\lambda\diam{\sigma([0,1])}. \]
\end{Lem}

\begin{proof}
Rescaling the metric, we may assume for simplicity that $\diam{\sigma([0,1])}=1$.

Let $V$ be a maximal $\lambda$-separated subset of $\sigma$ that contains $\sigma(0)$ and $\sigma(1)$. Define a graph $G=(V,E)$ with vertex set $V$ and edges $E = \{\{v,v'\} : v,v' \in V, d(v,v') \le \lambda\}$. We claim that $G$ is connected, i.e., for any $v,v' \in V$, there exist $v_1,\dots,v_n \in V$ such that $v_1=v$, $v_n=v'$, and $\{v_{i-1},v_{i}\} \in E$ for all $i\in\{2,\dots,n\}$. Towards a contradiction, suppose that $G'=(V',E')$ is a nonempty connected component of $G$ with $V'\subsetneq V$. Then, $\dist(V',V\setminus V') > \lambda$. But then $\sigma$ is contained in the union of the disjoint closed sets $G_1=\overline{N}(V',\lambda/2)$ and $G_2 = \overline{N}(V\setminus V',\lambda/2)$, which contradicts the fact that $\sigma$ is connected. 

Since $G$ is connected, there is a shortest path of vertices $\{v_1,\dots,v_n\}$ in $G$ with $v_1 = \sigma(0)$ and $v_n = \sigma(1)$.
Since $B(v_i,2^{-1}\lambda)$ are disjoint and contained in $B(\sigma(0),1)$, by the $(C_0,Q)$-regularity of $X$, 
\[ 2^{-Q}C_0^{-1}\lambda^Qn \le \sum_{i = 1}^n \mu(B(v_i,2^{-1}\lambda)) \le \mu(B(\sigma(0),1)) \le C_0, \]
so $n \le 2^QC_0^2\lambda^{-Q}$. By Lemma \ref{lem:BHR}, $Y$ is $p_0$-porous for some $p_0(C_0,Q,C_1,\alpha) \ge 1$ and for each $i\in\{2,\dots,n - 1\}$, there exists $v_i' \in B(v_i,3^{-1}\lambda)$ such that $B(v_i',(3p_0)^{-1}\lambda) \subset B(v_i,3^{-1}\lambda) \setminus Y$.
Note that since $\dist(\{\sigma(0),\sigma(1)\},Y) \ge c$, it suffices to choose $v_1' = v_1 = \sigma(0)$ and $v_n' = v_n = \sigma(1)$.

Fix $i\in\{1,\dots,n-1\}$. By \cite[{Corollary 3.4}]{HVZ24}, there is a curve $S_i$ lying in $B(v_i',2\lambda)$ joining $v_i'$ and $v_{i + 1}'$ such that $\length(S_i) \le \ell\lambda$ and for each $z \in S_i$, $\dist(z,Y) \ge \beta'\lambda$ for some $\ell(C_0,Q,p,C,c) \ge 1$ and $\beta'(C_0,Q,p,C,C_1,\alpha,c) \in (0,1/2)$. Let $S:[0,1]\to X$ denote the concatenation of curves $S_i$. Then $S$ satisfies $S(0) = \sigma(0)$, $S(1) = \sigma(1)$, 
\[ \length(S) \leq (n - 1)\ell\lambda \le 2^QC_0^2\ell\lambda^{1 - Q}, \]
and $S$ stays at least a distance $\beta'\lambda$ from $Y$. By \cite[Proposition 1.5]{HVZ24} there exists a constant $L(C,C_0,Q,p,C_1,\alpha,\lambda) \ge 1$ and there exists an $L$-bi-Lipschitz curve $\gamma:[0,1] \to X$, joining $S(0)$ to $S(1)$ and satisfying
\begin{align*}
\dist_H(\gamma([0,1]),S([0,1])) < \beta'\lambda^Q(2^{Q + 1}C_0^2\ell)^{-1}\diam{S([0,1])} & \le \beta'\lambda^Q(2^{Q + 1}C_0^2\ell)^{-1}\length(S) \\
& \le 2^{-1}\beta'\lambda.
\end{align*}
The curve $\gamma$ has length at most $L$ and
\[ \dist(\gamma([0,1]),Y) \geq \dist(S([0,1]),Y) - \dist_H(\gamma([0,1]),S([0,1])) \geq 2^{-1}\beta'\lambda. \]
Lastly, for each $t \in [0,1]$,
\[ \dist(\gamma(t),\sigma([0,1])) \le \dist(\gamma(t),S([0,1])) + \dist(S([0,1]),\sigma([0,1])) \le 2^{-1}\beta'\lambda + 3\lambda \le 4\lambda. \qedhere \]
\end{proof}

\section{A decomposition of compact totally disconnected sets} \label{Sec:ADecompositionOfUniformlyDisconnectedSets}

One natural property of compact, totally disconnected sets in Euclidean spaces is that they admit a \emph{defining sequence}, which is, roughly speaking, a way to coherently encapsulate clusters of the set; see for example Armentrout's work in $\R^3$ \cite{A66}. The main result of this section is Proposition \ref{Prop:DefiningSequence}, which shows that the same idea works in metric spaces with nice geometry.

\begin{Prop}\label{Prop:DefiningSequence}
Let $(X,d,\mu)$ be a complete, fully connected, unbounded, $(C_0,Q)$-regular, geodesic metric measure space supporting a $(Q,C)$-Poincar\'e inequality. Let $E$ be a compact, totally disconnected subset of $X$. There exists a dictionary $\W \subset \N^*$, and there exists a collection $\{U_{w} : w \in \W\}$ of bounded full sets such that the following conditions hold.
\begin{enumerate}
\item[(D1)] The empty word $\varepsilon \in \W$. If $w \in \W\setminus\{\varepsilon\}$, then $w^{\uparrow} \in \W$. For each $w \in \W$, there exists $n_w \in \N$ such that $wi \in \W$ if and only if $i \in \{1,\dots,n_w\}$.
\item[(D2)] For each $n \in \N\cup\{0\}$, $E \subset \bigcup_{\substack{w\in \W \\ |w|=n}} U_w$. Conversely, if $w\in \mathcal{W}$, then $U_w \cap E \ne \emptyset$.
\item[(D3)] If $wi,wj \in \W$ with $i\neq j$, then $U_{wi} \cap U_{wj} = \emptyset$.
\item[(D4)] If $w \in \W \setminus\{\varepsilon\}$, then $U_{w} \subset U_{w^{\uparrow}}$, $\diam U_{w} \le \tfrac12\diam U_{w^{\uparrow}}$, and $\dist(U_{w},\dee U_{w^\uparrow}) >0$.
\item[(D5)] If $w \in \W$, then $\dist(\dee U_w,E) > 0$.
\item[(D6)] If $w \in \W$, then the set
\[ \hat U_w := \ove{U_w} \setminus \bigcup_{i = 1}^{n_w} U_{wi} \]
is path connected.
\end{enumerate}
Moreover, if $E$ is $c$-uniformly disconnected for some $c > 1$, then there exists $\xi (C_0,Q,C,c) \ge 1$ such that for all $\delta \in (0,(2\xi)^{-1})$, there exists $N(C_0,\delta,Q) \in\N$ such that $\mathcal{W} \subset \{1,\dots,N\}^*$ and the following quantitative versions of (D4) and (D5) hold.
\begin{enumerate}
\item[(D4')]  If $w \in \W \setminus\{\varepsilon\}$, then $U_{w} \subset U_{w^{\uparrow}}$,
\begin{equation}\label{eq:diamUwi}
\delta\diam U_{w^{\uparrow}} \le \diam U_{w} \le \xi\delta\diam U_{w^{\uparrow}}, 
\end{equation}
and
\begin{equation}\label{eq:distUw-Uwi}
\dist(U_{w},\dee U_{w^\uparrow}) \ge \xi^{-1}\diam U_{w}. 
\end{equation}
\item [(D5')] If $w \in \W$, then $\dist(\dee U_w,E) \ge \xi^{-1}\diam U_w$.
\end{enumerate}
\end{Prop}

%We first need a few lemmas. 
For the rest of this section, we fix a complete, fully connected, unbounded, $(C_0,Q)$-regular, geodesic metric measure space $(X,d,\mu)$ supporting a $(Q,C)$-Poincar\'e inequality.

\begin{Lem}\label{Lem:UniformNeighborhood}
Given $r>0$ there exists $c_0(C_0,Q,r) \ge 1$ with the following property. If $U$ is a bounded, path connected subset of $X$ and $x,y \in U$, then there is a curve $\gamma$ lying in $N(U,r\diam U)$ joining $x$ with $y$ such that $\length(\gamma) \le c_0d(x,y)$.
\end{Lem}
	
\begin{proof}
Rescaling the metric, we may assume that $\diam U=1$. Set $N_r = N(U,r)$.

If $d(x,y) < r/6$, then we can join $x$ and $y$ by a geodesic $\gamma$, and this satisfies, for each $z \in \gamma$,
\[ \dist(z,X \setminus N_r) \ge \dist(\{x,y\},X \setminus N_r) - \tfrac{1}{2}\length(\gamma) \ge \tfrac{11}{12}r \]
and $\dist(z,\{x,y\}) \le \length(\gamma) = d(x,y) < r/6$.

If $d(x,y) \ge r/6$, then let $V$ be a $(r/6)$-net of $U$ containing $x$ and $y$. For each $u,v \in V$, fix a geodesic $\gamma_{u,v} = \gamma_{v,u}$ joining them and define
\[ G = \bigcup_{\substack{u,v \in V \\ d(u,v) \le r}} \gamma_{u,v}. \]

We show that the set $G$ is connected. Towards a contradiction, suppose that $Z$ is a nonempty path connected component of $G$ with $G \setminus Z \ne \emptyset$. Let $W = N(Z,2r)$. Note that $\partial W$ cannot be empty, since otherwise $W$ and $X \setminus W$ separate $X$. Thus, there is a point $\xi \in X$ with $\dist(\xi,Z) = 2r$. As $V$ is an $r$-net of $X$, there is a point $v \in V$ such that $d(\xi,v) \le r$. Necessarily $v \in G \setminus U$, since $\dist(v,Z) \ge \dist(\xi,Z) - d(\xi,v) \ge r$. We claim that there is a point $u \in V \cap Z$ such that $d(u,\xi) \le 5r$. To see this, choose $\zeta \in U$ such that $d(\xi,\zeta) = \dist(\xi,Z) = 2r$. If $\zeta \in V$, then taking $u = \zeta$, we have $d(u,\xi) = 2r$. If $\zeta \notin V$, then $\zeta$ lies on a geodesic joining two points $v_1,v_2 \in V$ for which $d(v_1,v_2) \le r$. Then either $d(\zeta,v_1) \le 3r$ or $d(\zeta,v_2) \le 3r$, say $d(\zeta,v_1) \le 3r$. Taking $u = v_1$, $d(u,\xi) \le d(u,\zeta) + d(\zeta,\xi) \le 5r$. Finally, $d(u,v) \le d(u,\xi) + d(\xi,v) \le 6r$, so in fact $\gamma_{u,v}$ is a path in $G$ joining $Z$ and $G \setminus Z$, a contradiction, which proves that $G$ is connected.

Since $G$ is connected there is a finite sequence $v_0,\dots,v_k$ of points in $V$ with $v_0 = x$, $v_k = y$, $d(v_{i - 1},v_i) \le r$ for each $i \in \{1,\dots,k\}$, and $v_i \ne v_j$ for distinct $i,j \in \{0,\dots,k\}$. Let $\gamma$ be the concatenation of the geodesic edges joining these vertices. Since $V$ is $(r/6)$-separated and contained in $U$, by $Q$-regularity of $X$,
\[ C_0^{-1}(r/12)^Q\card V \le \sum_{v \in V} \mu(B(v,r/12)) \le \mu(B(x,1)) \le C_0, \]
so $\card V \le 12^QC_0^2r^{-Q}$.
Thus
\[ \length(\gamma) \le kr \le r\card V \le 12^QC_0^2r^{1 - Q} \le 12^{Q + 1}C_0^2r^{-Q}d(x,y). \qedhere\]
%Further, for $i \in \{1,\dots,k\}$ and $z \in \gamma_{v_{i - 1},v_i}$,
%\[ \dist(z,X \setminus N_r) \ge \dist(\{v_{i - 1},v_i\},X \setminus N_r) - \frac{\length(\gamma_{v_{i - 1},v_i})}{2} \ge \frac{rD}{2}, \]
%while $\dist(z,\{x,y\}) \le \length(\gamma) \le C_0^2rD(6/r)^Q$.
\end{proof}

\begin{Lem}\label{Lem:LargeComponent}
There exists $K(C_0,Q,C) \ge 1$ such that if $U$ is a bounded and connected subset of $X$, then 
\begin{enumerate}[(a)]
\item $X \setminus U$ has exactly one component $V$ with $\diam V \ge K\diam U$, and
\item $\diam \dee U \le \diam U \le K\diam \dee U$.
\end{enumerate}
\end{Lem}

\begin{proof}
Let $C_1(C_0,Q) \ge 1$ be the constant from Lemma \ref{Lem:AnnulusModulus}, and let $\Phi:(0,\infty) \to (0,\infty)$ be the $Q$-Loewner function of $X$ given by Lemma \ref{Lem:Loewner}, which depends only on $C_0$, $Q$, and $C$.
Set
\[ K = \exp\left(\left(C_1^{-1}\Phi(4)\right)^{1/(1 - Q)}\right). \]

For (a), we first claim that there exists at least one component of $X\setminus U$ of diameter at least $K \diam{U}$. Indeed, if all components of $X\setminus U$ had diameter at most $K \diam{U}$, then it can be easily seen that $\diam{X} \leq (2K+1)\diam{U} < \infty$ which is false since $X$ is unbounded.

Now suppose for a contradiction that $V_1$ and $V_2$ are two different components of $X \setminus U$ with $\diam V_i \ge K\diam U$ for $i=1,2$. Fix $u \in U$ and note that $U \subset \overline B(u,\diam U)$. There also exist points $v_i \in V_i$ for $i=1,2$ such that $d(u,v_i) = K\diam U$. Let $\sigma_i$ be a curve in $V_i$ joining $v_i$ to $\dee U$ for $i=1,2$. Let $x_i$ be the first point that $\sigma_i$ intersects $\overline B(u,\frac12K\diam U)$ for $i=1,2$. Denote by $\gamma_i$ a subcurve of $\sigma_i$ joining $v_i$ to $x_i$ for $i=1,2$. Note that $\Delta(\gamma_1,\gamma_2) \le 2$, so $\Mod_Q(\gamma_1,\gamma_2) \ge \Phi(\Delta(\gamma_1,\gamma_2)) \ge \Phi(2)$. On the other hand, since $\gamma_1$ and $\gamma_2$ lie in different components of $X \setminus U$, any curve joining them must pass through $U$, and hence also $\overline B(u,\diam U)$. Thus $\Mod_Q(\gamma_1,\gamma_2) \le \Mod_Q(\ove B(u,\diam U),\gamma_1)$. Let $\Gamma$ be the collection of curves joining $X \setminus B(u,2K\diam U)$ to $\ove B(u,\diam U)$. By Lemma \ref{Lem:AnnulusModulus}, $\Mod_Q(\Gamma) \le C_1(\log(2K))^{1 - Q}$. Clearly $\Mod_Q(\ove B(u,\diam U),\gamma_1) \le \Mod_Q(\Gamma)$. By choice of $K$, 
\[ \Phi(2) \le \Mod_Q(\gamma_1,\gamma_2) \le \Mod_Q(\ove B(u,\diam U),\gamma_1) \le \Mod_Q(\Gamma) \le C_1(\log(2K))^{1 - Q} < \Phi(4), \]
which is a contradiction since $\Phi$ is decreasing.

For (b), since $\dee U \subset \ove U$ and $\diam U = \diam\ove U$, we have $\diam\dee U \le \diam\ove U = \diam U$. Next, we prove $\diam U \le K\diam\dee U$. Suppose for a contradiction that $\diam U > K\diam\dee U$. Choose $x_1,x_2 \in U$ such that $d(x_1,x_2) \ge \frac12\diam U$ and let $\gamma_1$ be a curve in $U$ joining $x_1$ and $x_2$. Note that there exists $y'\in X$ such that $\dist(y',U) \ge 2\diam U$ as otherwise we would have $\diam X < 5\diam U$, which is false. Since $X$ is connected, there is a point $y \in X$ such that $\dist(y,U) = 2\diam U$. Let $\sigma$ be a curve joining $y$ to a point on the boundary of $U$, and let $y_0$ be the first point $\sigma$ intersects $\ove N(U,\diam U)$. Set $\gamma_2$ to be a subcurve of $\sigma$ joining $y$ to $y_0$. Noting $\Delta(\gamma_1,\gamma_2) \le 4$, we have $\Mod_Q(\g_1,\g_2) \ge \Phi(4)$.

Fix now $z \in \dee U$.
Set $B_1 = \ove B(z,\dee U)$ and $B_2 = B(z,\diam U)$.
Note $\g_2 \subset X \setminus B_2$, as $\dist(\g_2,U) \ge \diam U$.
Let $\gamma$ be a curve joining $\gamma_1$ to $\gamma_2$.
Since $\gamma_1 \subset U$ and $\gamma_2 \subset X \setminus U$, $\gamma$ necessarily intersects $\dee U$, and hence $B_1$.
On the other hand, $\gamma$ intersects $\gamma_2$, hence it intersects $X \setminus B_2$.
Thus $\Mod_Q(\gamma_1,\gamma_2) \le \Mod_Q(B_1,X \setminus B_2)$.
By Lemma \ref{Lem:AnnulusModulus},
\[ \Mod_Q(B_1,X \setminus B_2) \le C_1\left(\log\left(\frac{\diam U}{\diam\dee U}\right)\right)^{1 - Q} < C_1(\log(K))^{1 - Q}. \]
Combining the previous inequalities and using the definition of $K$,
\[ \Phi(4) \le \Mod_Q(\gamma_1,\gamma_2) \le \Mod_Q(B_1,X \setminus B_2) < C_1(\log(K))^{1 - Q} = \Phi(4) \]
which gives a contradiction.
\end{proof}

\begin{Lem}\label{lem:Cantordecomp}
Let $E$ be a compact, totally disconnected metric space. Given $\delta \in (0,1)$, there exist mutually disjoint closed sets $E_1,\dots,E_k \subset E$ such that $E = E_1\cup\cdots \cup E_k$ and for all $i=1,\dots,k$, we have $\diam{E_i} \leq \delta \diam{E}$.
\end{Lem}

\begin{proof}
By Brouwer's theorem (see \cite[Theorem 7.8]{Kechris}), there exists an embedding $f:E \to \mathcal{C}$ of $E$ into the standard Cantor set $\mathcal{C}$. 

We create a dyadic subdivision $\{\mathcal{C}_{u} : u\in \{1,2\}^*\}$ as follows. Let $\phi_1,\phi_2 : \R \to \R$ be the similarities $\phi_1(x)=\frac13 x$ and $\phi_2(x) = \frac13 x + \frac23$. Define $\mathcal{C}_{\varepsilon} = \mathcal{C}$. Assuming that we have defined $\mathcal{C}_u$ for some $u \in \{1,2\}^*$, define $\mathcal{C}_{iu} = \phi_i(\mathcal{C}_u)$ for $i\in\{1,2\}$. 

By continuity of $f$, there exists $n\in\N$ such that $\sup_{w\in \{1,2\}^n}\diam{f^{-1}(\mathcal{C}_w)} \leq \delta \diam{E}$. To complete the proof, set $\{E_1,\dots,E_k\} = \{f^{-1}(\mathcal{C}_w) : w\in \{1,2\}^n\} \setminus \{\emptyset\}$.
\end{proof}

We are now ready to prove Proposition \ref{Prop:DefiningSequence}.

\begin{proof}[Proof of Proposition \ref{Prop:DefiningSequence}]
Rescaling the metric, we may assume that $\diam E \leq 1/2$. 

Fix a point $o \in E$, define $V_{\eps} = B(o,1)$, and note that $\dist(\partial V_{\varepsilon},E) \geq \frac12$. By Lemma \ref{Lem:LargeComponent}, $X \setminus V_{\varepsilon}$ has a unique unbounded component; denote the complement of this unbounded component by $U_{\varepsilon}$. Clearly $U_{\varepsilon}$ is connected and has connected complement, so it is full. Moreover, $\diam{U_{\varepsilon}} \leq K \diam{V_{\varepsilon}} \leq 2K$ and
\[ \dist(\partial U_{\varepsilon},E) \geq \dist(\partial V_{\varepsilon},E) \geq 1/2 \geq (4K)^{-1}\diam{U_{\varepsilon}}. \]

Assume now that for some $w\in \N^*$ we have defined a full set $U_w$ such that $E\cap U_w \neq \emptyset$ and $\dist(\partial U_w, E)>0$. Note that $E \cap U_w$ is compact and totally disconnected. By Lemma \ref{lem:Cantordecomp} we can partition $E\cap U_w = E_1 \cup\cdots\cup E_k$ such that $\diam{E_i} \leq (6K)^{-1} \diam{U_w}$. Let 
\[ \e = \min\left\{\tfrac12\dist(\partial U_w,E), \tfrac13 \min_{i\neq j}\dist(E_i,E_j),(6K)^{-1}\diam{U_w}\right\}.\]

Let $V_w$ be the $\e$-neighborhood of $E \cap U_w$ and denote by $\{V_{w,i} : i \in \mathcal{J}_w\}$ the components of $V_w$. Since each $V_{w,i}$ contains a ball of radius $\e$, by Ahlfors regularity of $X$, we have that $\mathcal{J}_w$ is a finite set. Moreover, 
\begin{equation}\label{eq:diamVwi2}
\e \leq \diam{V_{w,i}} \leq 2\e + (6K)^{-1}\diam{U_w} \leq (2K)^{-1}\diam{U_w}.
\end{equation}

Given $i \in \mathcal{J}_w$, by Lemma \ref{Lem:LargeComponent}, the set $X \setminus V_{w,i}$ has a unique unbounded component; denote the complement of this unbounded component by $\til{V}_{w,i}$. Clearly $V_{w,i} \subseteq \til{V}_{w,i}$. We claim that for any distinct $i,j \in \mathcal{J}_w$, either $\til{V}_{w,i} \subset \til{V}_{w,j}$, $\til{V}_{w,j} \subset \til{V}_{w,i}$, or $\til{V}_{w,i}\cap \til{V}_{w,j} =\emptyset$. If $V_{w,i} = V_{w,j}$, then $\til{V}_{w,i} = \til{V}_{w,j}$ and there is nothing to show. If $V_{w,i} \ne V_{w,j}$, then it is enough to show that either $(\til{V}_{w,i} \cup \til{V}_{w,j}) \setminus V_{w,i} \ne \emptyset$ or $(\til{V}_{w,i} \cup \til{V}_{w,j}) \setminus V_{w,j} \ne \emptyset$. For example, if ($\til{V}_{w,i} \cup \til{V}_{w,j}) \setminus V_{w,i} \ne \emptyset$, then $\til{V}_{w,i} \cup \til{V}_{w,j}$ is contained in a bounded component of $X \setminus V_{w,i}$, hence $\til{V}_{w,j} \subseteq \til{V}_{w,i}$. Since the union of two connected sets with nonempty intersection is connected, we need just show that, either $(\til{V}_{w,i} \cap \til{V}_{w,j}) \setminus V_{w,i} \ne \emptyset$, or $(\til{V}_{w,i} \cap \til{V}_{w,j}) \setminus V_{w,j} \ne \emptyset$. Suppose for contradiction that $(\til{V}_{w,i} \cap \til{V}_{w,j}) \setminus V_{w,i} = \emptyset$ and $(\til{V}_{w,i} \cap \til{V}_{w,j}) \setminus V_{w,j} = \emptyset$. Then
\[ \emptyset = ((\til{V}_{w,i} \cap \til{V}_{w,j}) \setminus V_{w,i}) \cup ((\til{V}_{w,i} \cap \til{V}_{w,j}) \setminus V_{w,j}) = (\til{V}_{w,i} \cap \til{V}_{w,j}) \setminus (V_{w,i} \cap V_{w,j}) = \til{V}_{w,i} \cap \til{V}_{w,j} \]
since $V_{w,i} \cap V_{w,j} = \emptyset$ as $\{V_{w,i} : i \in \mathcal{J}_w\}$ are disjoint.

By \eqref{eq:diamVwi2} and Lemma \ref{Lem:LargeComponent}, for each $i\in \mathcal{J}_w$
\begin{equation}\label{eq:diamtildeVwi1}
\e \leq \diam{V_{w,i}} \leq \diam{\til{V}_{w,i}} \leq K\diam{\partial \til{V}_{w,i}} \leq K\diam{\partial V_{w,i}} \leq K\diam{V_{w,i}} \leq \tfrac12\diam{U_w}
\end{equation}

Removing some indices from $\mathcal{J}_w$ if necessary, we obtain $\til{\mathcal{J}}_w \subset \mathcal{J}_w$ such that $\bigcup_{i\in \mathcal{J}_w}\til V_{w,i} \subset \bigcup_{i\in \til{\mathcal{J}}_w}\til V_{w,i}$ and distinct $i,j \in \til{\mathcal{J}}_w$ satisfy $\til{V}_{w,i} \ne \til{V}_{w,j}$. Let $\{U_{wi} : i \in \{1,\dots,n_w\}\}$ denote the maximal elements (with respect to inclusion) of $\{\til{V}_{w,i} : i \in \til{\mathcal{J}}_w\}$. For each $i \in \{1,\dots,n_w\}$, there exists $j \in \til{\mathcal{J}}_w$ such that $U_{wi} = \til{V}_{w,j}$. Since the complement of $\til{V}_{w,j}$ is connected, it is clear that $U_{wi}$ is connected and has connected complement, so $U_{wi}$ is full.

Let $\mathcal{W}$ be the set of all finite words $v \in \N^*$ such that $U_v$ has been defined. Property (D1) follows easily by induction on $|w|$. Properties (D2), (D3), and (D5) follow immediately from the construction of the domains $U_w$ and induction on $|w|$. Property (D4) follows from the construction of the domains $U_w$, induction on $|w|$, and \eqref{eq:diamtildeVwi1}.

Finally, for (D6), fix $w \in \W$. Since $X$ is fully connected and $U_w$ is a connected set containing the full set $U_{w1}$, $\ove{U_w} \setminus U_{w1}$ is path connected. Assume for some $k \in \{1,\dots,n_w\}$ that $\ove{U_w} \setminus \bigcup_{i = 1}^{k - 1} U_{wi}$ is path connected. Again since $X$ is fully connected and $\ove{U_w} \setminus \bigcup_{i = 1}^{k - 1} U_{wi}$ is a connected set containing the full set $U_{wk}$, $\ove{U_w} \setminus \bigcup_{i = 1}^k U_{wi}$ is path connected. Proceeding inductively, we obtain (D6).

Assume now that $X$ is $c$-uniformly disconnected. We prove the proposition for $\xi=12cK$ where $K>1$ is the constant from Lemma \ref{Lem:LargeComponent}. Fix $\delta \in (0,(2\xi)^{-1})$ and let $N = \lceil C_0^2\delta^{-Q} \rceil$.

Define $U_{\eps}$ as before. Assume that for some $w\in \{1,\dots,N\}^*$ we have defined a full set $U_w$ such that $\dist(\partial U_w,E) \geq (12c)^{-1} \diam{U_w}$. Let $V_w$ be a $(\delta\diam U_w)$-neighborhood of $E \cap U_w$ and denote by $\{V_{w,i} : i \in \mathcal{J}_w\}$ the components of $V_w$. We claim that
\begin{equation}
\delta\diam U_w \le \diam V_{w,i} \le 12c\delta\diam U_w, \qquad\text{for all $i\in \mathcal{J}_w$}. \label{eq:diamVwi}
\end{equation}
The lower bound of $\diam V_{w,i}$ is clear. Suppose for a contradiction that $\diam V_{w,i} > 12c\delta\diam V_w$ for some $i \in \mathcal{J}_w$. Choose $x,y \in V_{w,i}$ such that $d(x,y) \ge 2^{-1}\diam V_{w,i}$ and choose $x',y' \in V_{w,i} \cap E$ such that $d(x,x') < \delta\diam U_w$ and $d(y,y') < \delta\diam U_w$. Then
\[ d(x',y') \ge 2^{-1}\diam V_{w,i} - 2\delta\diam U_w > 2(3c - 1)\delta\diam U_w. \]
Since $V_{w,i}$ is connected and each point in $V_{w,i}$ is distance at most $\delta\diam U_w$ from a point in $E$, we can find a finite chain of points $\{x_0,\dots,x_k\} \subset E$, where $x_0 = x'$, $x_k = y'$, and $d(x_{j - 1},x_j) \le 2\delta\diam U_w$ for each $j \in \{1,\dots,k\}$. But this contradicts the equivalent definition of uniform disconnectedness given in Lemma \ref{Lem:EquivalentDefinition}. Thus, \eqref{eq:diamVwi} follows. Moreover, for all $i\in \mathcal{J}_w$,
\begin{equation}\label{eq:distVwi}
\dist(\dee V_{w,i},E) \ge \delta\diam U_w \ge (12c)^{-1}\diam V_{w,i}.
\end{equation}

For each $i\in \mathcal{J}_w$, let $B_i \subset V_{w,i}$ be a ball of radius $\delta\diam{U_w}$. The balls $B_i$ are mutually disjoint and contained in $U_w$. By Ahlfors regularity, 
\[ C_0^{-1}\delta^Q(\diam{U_w})^Q\card{\mathcal{J}_w} \leq \sum_{i \in \mathcal{J}_w} \mu(B_i) \leq \mu(U_w) \leq C_0 (\diam{U_w})^Q.\]
Therefore, $\card \mathcal{J}_w \le C_0^2\delta^{-Q} \le N$.

Define $\til{V}_{w,i}$ as before and recall that $V_{w,i} \subseteq \til{V}_{w,i}$ and that for any distinct $i,j \in \mathcal{J}_w$, either $\til{V}_{w,i} \subset \til{V}_{w,j}$, $\til{V}_{w,j} \subset \til{V}_{w,i}$, or $\til{V}_{w,i}\cap \til{V}_{w,j} =\emptyset$. By \eqref{eq:diamVwi} and Lemma \ref{Lem:LargeComponent}, for each $i\in \mathcal{J}_w$
\begin{align}\label{eq:diamtildeVwi2}
\delta\diam U_w \leq \diam{V_{w,i}} \leq \diam{\til{V}_{w,i}} \leq K\diam{\partial \til{V}_{w,i}} \leq K\diam{\partial V_{w,i}} &\leq K\diam{V_{w,i}}\\ 
&\leq 12cK\diam{U_w}.\notag
\end{align}

Define now $\{U_{wi} : i\in \til{\mathcal{J}}_w\}$ as before and recall that each $U_{wi}$ is full. Let $\mathcal{W}$ be the set of all finite words $v \in \N^*$ such that $U_v$ has been defined. The fact that $\mathcal{W} \subset \{1,\dots,N\}^*$ follows by an induction on $|w|$. Inequality \eqref{eq:diamUwi} follows from \eqref{eq:diamtildeVwi2} while (D5') follows from \eqref{eq:distVwi} and the fact that for all $w\in \{1,\dots,n_w\}$ there exists $j\in \mathcal{J}_w$ such that $\partial U_{w,i} \subset \partial V_{w,j}$. 

For \eqref{eq:distUw-Uwi}, recall that $\xi=12cK$ and that $\delta \in (0,(2\xi)^{-1})$. By (D5') and \eqref{eq:diamUwi},
\begin{align*}
\dist(U_{w},\dee U_{w^{\uparrow}}) \ge \dist(\dee U_{w^{\uparrow}},E) - \delta\diam U_{w^{\uparrow}} \ge (\xi^{-1} - \delta)\diam U_{w^{\uparrow}} &\ge (2\xi)^{-1}\diam U_{w^{\uparrow}}\\ 
&\ge \xi^{-1}\diam U_{w}.
\end{align*}
This completes the proof of the proposition.
\end{proof}

%Henceforth, we say that a word $w\in\W$ is the \emph{parent} of a word $u\in\W$ if there exists $i\in \{1,\dots,n_w\}$ such that $u=wi$.

%In the sequel, it would be useful to have a decomposition of the entire space $X$ rather than just around $E$. We finish this section by extending our decomposition $\{U_w: w\in \mathcal{W}\}$ to the entire space $X$ as follows.

%Set $U^0 = U_{\varepsilon}$ and recall that $B(o,1) \subset U_{\eps}$ for some $o\in E$. Assume that for some $n\in \N$ we have defined a full set $U^{1-n}$ such that $B(o,(4K)^{n-1}) \subset U^{1-n}$, $\diam{U^{1-n}} \leq 2K(4K)^{n-1}$ and $\dist(\partial U^{1-n}, E) \geq \frac12(4K)^{n-1}$. Define $V^{-n} = B(o, (4K)^n)$ and by Lemma \ref{Lem:LargeComponent} we have that $X \setminus V^{-n}$ has a unique unbounded component. Let $U^{-n}$ be the complement of this unbounded component. Working as in (D6), we can show that $U^{-n}$ is full, $\overline{U^{-n}} \setminus U^{1-n}$ is path-connected, $U^{-n}$ contains $B(o,(4K)^n)$ and 
%\[ \diam{U^{-n}} \leq K\diam{B(o,(4K)^n)} \leq 2K(4K)^n, \quad \dist(\partial U^{-n}, E) \geq \frac12(4K)^{n}.\]
%
%Proceeding inductively, we obtain domains $U_{\varepsilon} = U^0 \subset U^{-1} \subset \cdots$ with
%\[ X = \left(\bigcup_{n\in\N} \overline{U^{-n}} \setminus U^{1-n}\right) \cup \bigcup_{n\in\N}\bigcup_{\substack{w\in \mathcal{W}\\ |w|=n}} \left( \overline{U_w} \setminus \bigcup_{i=1}^{n_w}U_{wi}\right).\]

\section{Co-uniformity of uniformly disconnected sets} \label{Sec:Co-uniformity}

The goal of this section is to prove Theorem \ref{Thm:Uniformity}. For this section, fix a complete, fully connected, unbounded, $(C_0,Q)$-regular metric measure space $X$ supporting a $(Q,C)$-Poincar\'e inequality. By Lemmas \ref{Lem:Quasiconvex} and \ref{Lem:Preserve} we may assume that $(X,d)$ is geodesic. Fix also a $c$-uniformly disconnected set $E\subset X$. Without loss of generality we may assume that $\diam{E}\leq 1/2$. 

Let $\xi \ge 1$ and $N \in \N$ be the constants of Proposition \ref{Prop:DefiningSequence}.
Fix $\delta = (4\xi)^{-1}$ and let $\W \subseteq \{1,\dots,N\}^*$ and $\{U_w : w \in \W\}$ be the dictionary and decomposition of Proposition \ref{Prop:DefiningSequence}.
%Currently, this decomposition begins with a small neighborhood of $U_{\eps}$, but 
While the ``top set'' $U_{\eps}$ is only a bounded neighborhood of $E$, it is useful for the proof of Theorem \ref{Thm:Uniformity} to have a decomposition of the entirety of $X$.

%We construct such an extension of the decomposition below, which is similar to (but much simpler than) the construction in the proof of Proposition \ref{Prop:DefiningSequence}.

%\begin{Lem}\label{Lem:Exhaustion}
%There exists a collection of full, open sets $\{U_w : w \in \W'\}$ such that the following conditions hold.
%\begin{enumerate}[(E1)]
%\item If $w \in \W'$, then $w^{\uparrow} \in \W'$.
%For each $w \in \W'$, there exists $u \in \W'$ such that $u^{\uparrow} = w$. Furthermore, if $|w| \ge 0$, there exists $n_w \in \N$ such that $wi \in \W'$ if and only if $i \in \{1,\ldots,n_w\}$.
%\item For each $n \in \Z$, $E \subset \bigcup_{\substack{w \in \W' \\ |w| = n}} U_w$. Conversely, if $w \in \W'$, then $U_w \cap E \ne \emptyset$.
%\item If $wi,wj \in \W'$ with $i\neq j$, then $U_{wi} \cap U_{wj} = \emptyset$.
%\item If $w \in \W'$, then $U_w \subset U_{w^{\uparrow}}$, $(2\xi)^{-2}\diam U_{w^{\uparrow}} \le \diam U_w \le 2^{-2}\diam U_{w^{\uparrow}}$, and
%\[ \dist(U_w,\dee U_{w^{\uparrow}}) \ge \xi^{-1}\diam U_w. \]
%\item If $w \in \W'$, then $\dist(\dee U_w,E) \ge \xi^{-1}\diam U_w$.
%\item $\hat U_w := \ove U_w \setminus \bigcup_{\substack{u \in \W' \\ u^{\uparrow} = w}} U_u$ is path connected.
%\end{enumerate}
%\end{Lem}

\begin{Lem}\label{Lem:Exhaustion}
There exists a collection of full, open sets $U_{\varepsilon}=U_0 \subset U_{-1}\subset U_{-2}\subset \cdots$ such that the following conditions hold.
\begin{enumerate}[(E1)]
\item For each $n\in\N$, $(2\xi)^{-2}\diam U_{-n} \le \diam U_{1-n} \le \frac14\diam U_{-n}$, and
\[ \dist(U_{1-n},\dee U_{-n}) \ge \xi^{-1}\diam U_{1-n}. \]
\item For each $n\in\N$, $\dist(\dee U_{1-n},E) \ge \xi^{-1}\diam U_{1-n}$.
\item For each $n\in\N$, $\hat U_{-n} := \ove U_{-n} \setminus U_{1-n}$ is path connected.
\item We have $X = \bigcup_{n\in\N}U_{-n}$.
\end{enumerate}
\end{Lem}

\begin{proof}
Set $U_0=U_{\varepsilon}$. Properties (E1)--(E3) are vacuously true. Assume that for some integer $n\geq 0$, we have defined a full set $U_{-n}$ such that (E1)--(E3) hold for $n$. Let $V_{-n - 1}$ be the $(4\xi\diam U_{-n})$-neighborhood of $U_{-n}$.
By Lemma \ref{Lem:LargeComponent}, there is a unique unbounded component of $X \setminus V_{-n - 1}$; denote the complement of this unbounded component by $U_{-n - 1}$.
Clearly $U_{-n - 1}$ is connected and has connected complement, so it is full.
Moreover, recalling that $\xi = 12cK$ in the proof of Proposition \ref{Prop:DefiningSequence},
\[ (2\xi)^{-2}\diam U_{-n - 1} \le \diam U_{-n} \le 2^{-2}\diam U_{-n - 1} \]
and 
\[ \dist(\dee U_{-n - 1},E) \ge \dist(\dee U_{-n - 1},U_{-n}) \ge 4\xi\diam U_{-n} \ge \xi^{-1}\diam U_{-n - 1}.\]
Lastly, since $X$ is fully connected and $U_{-n - 1}$ is a connected set containing a neighborhood of the full set $U_{-n}$, we have $\hat U_{-n - 1}$ is connected and we obtain (E1)--(E3) for $n+1$.

Proceeding inductively, we define sets $\{U_{-n}\}_{n\in\N}$. Given an integer $n\geq0$, by (E1),
\[ \diam{U_{0}} + \cdots + \diam{U_{-n}} \geq \diam{U_{\varepsilon}} (1+\cdots+4^n) \geq 4^n\diam{U_{\varepsilon}}.\]
Fix $o\in U_{\varepsilon}$. Then for all integers $n\geq 0$, $B(o, 4^n\diam{U_{\varepsilon}}) \subset U_{-n-1}$ and (E4) follows.
\end{proof}

We now proceed to the proof of Theorem \ref{Thm:Uniformity}. To simplify the exposition, when we write \texttt{data}, we mean the constants $C_0$, $Q$, $C$, and $c$.

We also extend the word-parent function by setting $\eps^{\uparrow} = -1$ and by setting $(-n)^{\uparrow} = -n - 1$ for each $n\in\N$. We also extend the word-length function by setting $|-n| = -n$ for all $n\in\N$. Finally, we set $\W' = \{-n : n \in \N\} \cup \W$.

\begin{proof}[{Proof of Theorem \ref{Thm:Uniformity}}]
Fix $x,y \in X \setminus E$. There exists a maximal $u \in \W'$ (with respect to $|\cdot|$) such that $x,y \in U_u$. Further, there exist $v,w \in \W'$ such that $x \in \hat U_v$ and $y \in \hat U_w$. Clearly $|u| \le \min\{|v|,|w|\}$. Without loss of generality, we may assume $|v| \le |w|$. There are a few cases to consider.

\emph{Case 1.} Suppose that $d(x,y) < (2\xi)^{-5}\diam U_u$.
We claim that $|v| \le |u| + 1$ and $|w| \le |u| + 1$.
To see this, suppose for a contradiction that $|w| - |u| > 1$.
Let $u_1,u_2 \in \W'$ such that $u_1^{\uparrow} = u$, $u_2^{\uparrow} = u_1$, and $y \in U_{u_2}$.
By Proposition \ref{Prop:DefiningSequence}(D4') and Lemma \ref{Lem:Exhaustion}(E1),
\[ d(x,y) \ge \dist(\dee U_{u_1},U_{u_2}) \ge \xi^{-1}\diam U_{u_2} \ge 2^{-2}\xi^{-3}\diam U_{u_1} \ge 2^{-4}\xi^{-5}\diam U_u, \]
which contradicts the assumption $d(x,y) < (2\xi)^{-5}\diam U_u$.

Let $\Gamma$ be a geodesic joining $x$ to $y$.
Clearly $\length(\Gamma) = d(x,y)$.
By Proposition \ref{Prop:DefiningSequence}(D5'), Lemma \ref{Lem:Exhaustion}(E2), and since $|w| \le |u| + 1$, for any $z \in \Gamma$ and $w_0 \in \W'$ with $w_0^{\uparrow} = w$,
\begin{align*}
\dist(z,E) & \ge \dist(\dee U_{w_0},E) - \length(\Gamma) \\
& \ge \xi^{-1}\diam U_{w_0} - d(x,y) \\
& > 2^{-2}\xi^{-3}\diam U_w - (2\xi)^{-5}\diam U_u \\
& \ge (2\xi)^{-5}\diam U_u \\
& > d(x,y) \\
& = \length(\Gamma) \\
& \ge \dist(z,\{x,y\}).
\end{align*}

\emph{Case 2.} Suppose that $d(x,y) \ge (2\xi)^{-5}\diam U_u$ and $u = v = w$.
By Lemma \ref{Lem:UniformNeighborhood}, there is a curve $\Gamma$ joining $x$ and $y$ lying in the neighborhood $N(\hat U_u,(2\xi)^{-3}\diam U_u)$ with length at most $c_0d(x,y)$, where $c_0(\texttt{data}) \ge 1$. Then 
\[ d(x,y) \le \length(\Gamma) \le c_0d(x,y) \le c_0\diam U_u\] 
and for any $z \in \Gamma$ and $u_0 \in \W'$ such that $u_0^{\uparrow} = u$, by Proposition \ref{Prop:DefiningSequence}(D4',D5') and Lemma \ref{Lem:Exhaustion}(E1,E2),
\[ \dist(z,E) \ge \dist(\dee U_{u_0},E) - (2\xi)^{-3}\diam U_u \ge \xi^{-1}\diam U_{u_0} - (2\xi)^{-3}\diam U_u \ge (2\xi)^{-3}\diam U_u. \]

\emph{Case 3.} Suppose that $d(x,y) \ge (2\xi)^{-5}\diam U_u$ and $u = v \ne w$.
For each $i \in \{1,\ldots,|w| - |u|\}$, let $w_i \in \W'$ such that $w_1 = w$, $w_i^{\uparrow} = w_{i + 1}$, and $w_{|w| - |u|} = u$.
Let $z_i \in \dee U_{w_i}$ for each $i \in \{1,\ldots,|w| - |u| - 1\}$.
By Case 2, there exists a curve $\gamma_i$ in $N(\hat U_{w_i},(2\xi)^{-3}\diam U_{w_i})$ joining $z_{i - 1}$ and $z_i$ with length at most $c_0d(z_{i - 1},z_i)$ for each $i \in \{2,\ldots,|w| - |u| - 2\}$, there exists a curve $\gamma_1$ in $N(\hat U_w,(2\xi)^{-3}\diam U_w)$ joining $x$ and $z_1$ with length at most $c_0d(x,z_1)$, and there exists a curve $\gamma_{|w| - |u|}$ in $N(\hat U_u,(2\xi)^{-3}\diam U_u)$ joining $z_{|w| - |u| - 1}$ to $y$ with length at most $c_0d(z_{|w| - |u| - 1},y)$.
Define $\Gamma$ to be the concatenation of these curves, which joins $x$ to $y$.
By Proposition \ref{Prop:DefiningSequence}(D4') and Lemma \ref{Lem:Exhaustion}(E1),
\begin{align*}
\length(\Gamma) & = \sum_{i = 1}^{|w| - |u|} \length(\gamma_i) \\
& \lsim_{c_0} d(x,z_1) + d(z_{|w| - |u| - 1},y) + \sum_{i = 2}^{|w| - |u| - 1} d(z_{i - 1},z_i) \\
& \le \diam U_w + \diam U_u + \sum_{i = 2}^{|w| - |u| - 1} \diam U_{w_i} \\
& < 2\diam U_u \\
& \le 2^6\xi^5d(x,y).
\end{align*}
Let now $z \in \Gamma$ and suppose $z \in \gamma_i$ for some $i \in \{2,\ldots,|w| - |u| - 1\}$ (the cases where $z \in \gamma_1$ or $z \in \gamma_{|w| - |u|}$ are similar).
Since $\gamma_i$ lies in $N(\hat U_{w_i},(2\xi)^{-3}\diam U_{w_i})$, by Proposition \ref{Prop:DefiningSequence}(D4',D5') and Lemma \ref{Lem:Exhaustion}(E1,E2), for any $v_0 \in \W'$ such that $v_0^{\uparrow} = w_i$,
\begin{align*}
\dist(z,E) & \ge \dist(\dee U_{v_0},E) - (2\xi)^{-3}\diam U_{w_i} \\
& \ge \xi^{-1}\diam U_{v_0} - (2\xi)^{-3}\diam U_{w_i} \\
& \ge (2\xi)^{-3}\diam U_{w_i},
\end{align*}
while
\[ \dist(z,\{x,y\}) \le \dist(\gamma_i,\{x,y\}) + \length(\gamma_i) \le \diam U_{w_i} + c_0d(z_{i - 1},z_i) \le (1 + c_0)\diam U_{w_i}. \]
\emph{Case 4.} Suppose that $d(x,y) \ge (2\xi)^{-5}\diam U_u$ and $u \ne v$.
The proof is a two-sided argument similar to that in Case 3.
Let $u_1,u_2 \in \W'$ such that $u_1^{\uparrow} = u_2^{\uparrow} = u$ and $x \in U_{u_1}$ and $y \in U_{u_2}$.
Fix $x' \in \dee U_{u_1}$ and $y' \in \dee U_{u_2}$.
Let $\gamma_1$ be the curve from Case 3 joining $x$ to $x'$, let $\gamma_2$ be the curve from Case 3 joining $y'$ to $y$, and let $\gamma_0$ be the curve in $N(\hat U_u,(2\xi)^{-3}\diam U_u)$ from Case 2 joining $x'$ to $y'$.
Let $\Gamma$ be the concatenation of the curves $\gamma_0$, $\gamma_1$, and $\gamma_2$.
We have
\[ \length(\Gamma) = \length(\gamma_0) + \length(\gamma_1) + \length(\gamma_2) \lsim_{c_0,\xi} \diam U_u \lsim_{\xi} d(x,y). \]
For a point $z \in \Gamma$, we have $\dist(z,\{x,y\}) \lsim_{c_0,\xi} \dist(z,E)$ from Cases 2 and 3.
\end{proof}

\section{Preparation for the proofs of Theorem \ref{thm:main1} and Theorem \ref{Thm:main2}} \label{Sec:Quasisymmetric1}

Fix, for the rest of the paper, a metric measure space $(X,d,\mu)$ which is complete, fully connected, $(C_0,Q)$-regular, and supports a $(Q-1,C)$-Poincar\'e inequality for some $C_0,C>1$, and $Q>2$. By Lemmas \ref{Lem:Quasiconvex} and \ref{Lem:Preserve} we may assume that $(X,d,\mu)$ is geodesic. Fix also a totally disconnected set $E\subset X$ and a number $c\geq 1$. Let $\xi \ge 1$ be the constant from Proposition \ref{Prop:DefiningSequence} associated to the number $c$. Set $\delta = (4\xi)^{-1}$, and let $N\in\N$, $\W \subset \N^*$, and $\{U_w : w \in \W\}$ be as in Proposition \ref{Prop:DefiningSequence} associated to the set $E$ and the number $\delta$. 

To simplify the exposition, we make the assumption that the only word in $\W$ with length one is $1$ (so $E$ is completely contained in $U_1$). This can easily be arranged by rescaling the metric. 
%To see that one can make this assumption, let $B$ be a ball of radius $\delta^{-1}\diam U_{\eps}$ centered at a point of $U_{\eps}$. By Lemma \ref{Lem:LargeComponent}, there is only one component $V$ of $X \setminus B$ with $\diam V \ge K\diam B$. Set $U$ to be the complement of $V$ and define $\cal U_{\eps} = U$, $\cal U_1 = U_{\eps}$, and $\cal U_{1w} = U_w$ for all $w \in \W$.For simplicity, we still use the notation $\{U_w : w \in \W\}$ for this new collection. 
The use of this assumption is that given a word $w \in \W$, we construct curves that lie in a neighborhood of the children of $U_w$, and join the boundaries of the grandchildren of $U_w$. Since we look two generations ahead, it is useful to have $U_{\eps}$ and $U_1$ be ``extraneous'' so that we do not have to consider these cases separately.

\subsection{Basic notation}
We establish some notation which we use throughout the rest of the paper. 
\begin{itemize}
\item For each $w\in\W$, by Proposition \ref{Prop:DefiningSequence}, there exists $n_w \in \N$ such that $wi \in \mathcal{W}$ if and only if $i \in \{1,\dots,n_w\}$.
\item For each $w \in \W$, define $\displaystyle \hat U_w = U_w \setminus \bigcup_{i = 1}^{n_w} U_{wi}$.
\item For each $w \in \W$, define 
\begin{align*}
\e_w = (4\diam{U_w})^{-1}\min \Bigg\{\dist(U_{w}, \partial U_{w^{\uparrow}}), & \min_{i=1,\dots,n_w}\dist(\partial U_w, U_{wi}), \\
& \min_{\substack{i=1,\dots,n_w\\ j=1,\dots,n_{wi}}} \dist(\partial U_{wi},U_{wij}),\min_{i = 1,\ldots,n_w} \diam U_{wi} \Bigg\}
\end{align*}
with the understanding that $\dist(\partial U_{w}, \partial U_{w^{\uparrow}})$ is dropped if $w =\varepsilon$. If $E$ is $c$-uniformly disconnected, then we set $\e_w = (4\xi)^{-3}$.
\item For each $w \in \W$, define $\mathcal{N}_w = N(\hat U_w,\e_w\diam{U_w})$.
\item For each $w\in \W\setminus \{\varepsilon\}$, define
\[ \lambda_w = \min\left\{ \max_{i,j = 1,\dots,n_w}\frac{\diam{U_{wi}}}{\diam{U_{wj}}}, \max_{i = 1,\dots,n_w} \frac{\diam{U_{wi}}}{\dist(\partial U_{wi},E)} \right\}.\]
If $E$ is $c$-uniformly disconnected, then set $\lambda_w = \xi$.
\end{itemize}

Note that if $E$ is $c$-uniformly disconnected, and if $w\neq \varepsilon$, then a simple computation involving \eqref{eq:diamUwi} and \eqref{eq:distUw-Uwi} shows that $\e_w \geq (4\xi)^{-3}$ and $\lambda_w \geq \xi$, which justifies the choices $\e_w = (4\xi)^{-3}$ and $\lambda_w=\xi$.
%\[  (4\diam{U_w})^{-1}\dist(U_{w}, \partial U_{w^{\uparrow}}) \geq (4\xi)^{-1}\]
%while for $i \in \{1,\dots,n_w\}$ and $j \in \{1,\dots,n_{wi}\}$, by \eqref{eq:diamUwi} and \eqref{eq:distUw-Uwi},
%\begin{align*}
%(4\diam{U_w})^{-1}\dist(\partial U_{w}, U_{wi}) & \geq (4\xi)^{-2}, \\
%(4\diam{U_w})^{-1}\dist(\partial U_{wi}, U_{wij}) & \geq (4\xi)^{-3}, \text{ and } \\
%(4\diam U_w)^{-1}\diam U_{wij} & \ge (4\xi)^{-3},
%\end{align*}
%which justifies the choice $\e_w = (4\xi)^{-3}$. Moreover, for any $i,j \in \{1,\dots,n_w\}$, $\diam{U_{wi}} \leq \xi\diam{U_{wj}}$ by \eqref{eq:diamUwi}, and $\diam{U}_{wi} \leq \xi \dist(\partial U_{wi},E)$ by Proposition \ref{Prop:DefiningSequence}(D5'). Therefore, if $E$ is $c$-uniformly disconnected, we can 

\begin{Lem}\label{lem:distofhats}
Let $w,u \in \W$ be two distinct words that do not have a common parent and one is not the parent of the other. Then
\[ \dist(\mathcal{N}_{w}, \mathcal{N}_{u}) \geq \e_{w}\diam{U_{w}}+\e_{u}\diam{U_{u}}.\]
\end{Lem}

\begin{proof}
The assumptions on $w,u$ imply that
\[ \dist(\hat U_{w},\hat U_{u}) \geq \dist\left (\hat U_{w}, \partial U_{w^{\uparrow}}\cup \bigcup_{i=1}^{n_w}\bigcup_{j=1}^{n_{wi}}U_{wij}\right) \geq 4\e_w\diam{U_w}. \]
Therefore,
\begin{align*} 
\dist(\mathcal{N}_{w}, \mathcal{N}_{u}) &\geq \dist(\hat U_{w},\hat U_{u}) - \e_{w}\diam{U_{w}} - \e_{u}\diam{U_{u}} \geq \e_{w}\diam{U_{w}} + \e_{u}\diam{U_{u}}. \qedhere
\end{align*}
\end{proof}

\subsection{Entries and exits}
We define here for each $w\in\W$ an ``entry'' to $U_w$ and an ``exit''. The following lemma allows us to do so.

\begin{Lem}\label{Lem:Entry}
Given $n\in\N$ and $\lambda > 1$, there is a constant $\rho(n,C_0,Q,C,\lambda) \in (0,1)$ satisfying the following property. Let $W_1,\dots,W_n \subset X$ be mutually disjoint domains such that for each distinct $i,j \in \{1,\dots,n\}$, $\diam{W_i} \leq \lambda \diam{W_j}$ and for each $i \in \{1,\dots,n\}$ there exists a ball $B_i \subset W_i$ of radius $\lambda^{-1}\diam{W_i}$. Then for each $i \in \{1,\dots,n\}$ there exists $x_i \in \dee W_i$ such that if $i,j \in \{1,\dots,n\}$ are distinct, then $d(x_{i},x_{j}) \ge \rho\diam{W_{i}}$.
\end{Lem}

\begin{proof}
The construction of constant $\rho$ and points $x_1,\dots,x_n$ is done in an inductive fashion. 

Fix $x_1 \in \dee W_1$ and let $\rho_1=1/2$. Suppose that for some $k \in \{1,\dots,n-1\}$, we have defined $ \rho_k \in (0,1)$ and points $x_i \in \dee W_i$ for $i \in \{1,\dots,k\}$ such that $d(x_i,x_j) \ge \rho_k\diam{W_i}$ for distinct $i,j \in \{1,\dots,k\}$. By the assumptions of the lemma, there exists $u \in W_{k+1}$ such that $B(u,\lambda^{-1}\diam{W}_{k+1}) \subset W_{k+1}$. Fix $v \in X$ such that $\dist(v,W_{k+1}) = \diam{W}_{k+1}$ and set $B = \ove B(u,(3\lambda)^{-1}\diam{W_{k+1}})$ and $B' = \ove B(v,(3\lambda)^{-1}\diam{W_{k+1}})$. Note that
\[ \Delta(B,B') \le \frac{d(u,v)}{\diam{B}} \le \frac{\diam{W_{k+1}} + \dist(v,W_{k+1})}{\diam{B}} \le 6\lambda, \]
so by Lemma \ref{Lem:Loewner} we have $\Mod_Q(B,B') \ge \Phi(6\lambda)$ where $\Phi$ is a decreasing homeomorphism of $(0,\infty)$ depending only on $C_0,Q$, and $C$.

Let $C_1(C_0,Q) \ge 1$ be the constant from Lemma \ref{Lem:AnnulusModulus}. Let $a_{k+1} \in (0,(3\lambda)^{-1})$ so that
\[ C_1\left(\log\left(\frac{\lambda^{-1}}{3a_{k+1}}\right)\right)^{1 - Q} \le \frac{\Phi(6\lambda)}{2k}. \]
For $i \in \{1,\dots,k\}$ define $B_i = \ove B(x_i,a_{k+1}\diam{W_{k+1}})$ and $B_i' = B(x_i,(3\lambda)^{-1}\diam{W_{k+1}})$. Let $\Gamma_0$ denote the family of curves joining $B$ and $B'$, and for each $i \in \{1,\dots,k\}$, let $\Gamma_i$ denote the family of curves in $\Gamma_0$ that intersect $B_i$. Note that for each $i\in \{1,\dots,k\}$,
%\[ d(x_i, v) \geq d(v,u)-d(x_i,u) \geq d(v,u) - \sum_{j=1}^k \diam{W_j} \geq \diam{W_{k+1}}\]
%and 
$d(x_i,u) \geq \lambda^{-1}\diam{W_{k+1}}$.
Hence each curve $\gamma \in \Gamma_i$ has a subcurve which joins $B_i$ and $X \setminus B_i'$, so by Lemma \ref{Lem:AnnulusModulus},
\[ \Mod_Q(\Gamma_i) \le \Mod_Q(B_i,X \setminus B_i') \le C_1\left(\log\left(\frac{\lambda^{-1}}{3a_{k+1}}\right)\right)^{1 - Q} \le \frac{\Phi(6\lambda)}{2k}. \]
Therefore,
\[ \Mod_Q(\Gamma_0 \setminus (\Gamma_1 \cup \cdots \cup \Gamma_{k})) \ge \Mod_Q(\Gamma_0) - \sum_{i = 1}^{k} \Mod_Q(\Gamma_i) \ge \tfrac12\Phi(6\lambda). \]
Namely, $\Mod_Q(\Gamma_0 \setminus (\Gamma_1 \cup \cdots \cup \Gamma_{k})) > 0$, so there is a curve $\gamma$ which joins $B$ and $B'$ and does not intersect $B_i$ for any $i \in \{1,\dots,k\}$. Since $B \subset W_{k + 1}$ and $B' \cap W_{k + 1} = \emptyset$, there exists a point $x_{k + 1}$ where the curve $\gamma$ intersects $\dee W_{k + 1}$. Then $d(x_{k + 1},x_i) \ge a_{k + 1}\diam{W_{k+1}}$ for each $i \in \{1,\dots,k\}$. To complete the inductive step, set $\rho_{k+1} = \min \{\rho_k,\lambda^{-1}a_{k+1}\}$ and note that $d(x_{i},x_j) \ge \rho_{k+1}\diam{W_{i}}$ for all distinct $i,j \in \{1,\dots,k+1\}$.

Proceeding inductively, we find $\rho_n(C_0,Q,C,\lambda,n) \in (0,1)$ and points $x_i \in \dee W_i$ for each $i \in \{1,\dots,n\}$ such that $d(x_i,x_j) \ge \rho_n\diam{W_i}$ when $i \ne j$.
\end{proof}

%In view of the notation above, Lemma \ref{Lem:Entry} provides us with points $x_w,y_w \in \dee U_w$; one of them considered as the ``entry point'' of $U_w$ and the other as the ``exit point'' of $U_w$. We use these points to construct curves which traverse all sets in $\{U_w : w \in \W\}$. 

If $x\in U_{wij}$ for some $i\in\{1,\dots,n_w\}$ and $j\in\{1,\dots,n_{wi}\}$, then $B(x,\lambda_{w}\diam{U_{wi}}) \subset U_{wi}$. Therefore, for each $w\in \W$, we can apply Lemma \ref{Lem:Entry} with $\lambda = \lambda_w$ and $W_i = U_{wi}$, $i=1,\dots,n_w$, and we obtain a constant $\rho_w \in (0,1)$ 
%depending only on $C_0$, $Q$, $p$, and $C$ along with 
and points $x_{wi} \in \dee U_{wi}$, $i \in \{1,\dots,n_w\}$, such that 
\[ d(x_{wi},x_{wj}) \ge \rho_w\max\{\diam U_{wi},\diam{U_{wj}}\}, \qquad\text{for distinct $i,j \in \{1,\dots,n_w\}$.}\] 
For each $i \in \{1,\dots,n_w\}$, fix a point $y_{wi} \in \dee U_{wi}$ such that $d(x_{wi},y_{wi}) = \frac13\rho_w\diam U_{wi}$ and note that this implies $d(x_{wi},y_{wj}) \ge \frac13\rho_w\diam U_{wi}$ and $d(y_{wi},y_{wj}) \ge \frac13\rho_w\diam U_{wi}$ for distinct $i,j \in \{1,\dots,n_w\}$. 

\begin{Rem}\label{rem:rho}
If $E$ is $c$-uniformly disconnected, then $\rho_w$ is independent of $w$ and only depends on $C_0$, $Q$, $C$, and $c$.
\end{Rem}

Lastly, for each $w\in\W$ and $i\in\{1,\dots,n_w\}$, define 
\[ X_{wi} = B(x_{wi},\tfrac1{12}\e_{wi}\rho_{w}\diam U_{wi})\qquad\text{and}\qquad Y_{wi} = B(y_{wi},\tfrac1{12}\e_{wi}\rho_{w}\diam U_{wi})\] 
We view $X_{wi}$ as the ``entry'' of the set $U_{wi}$ and we view $Y_{wi}$ as the ``exit'' of the set $U_{wi}$. 

\begin{Lem}
Let $w,u \in \W$, $i\in\{1,\dots,n_w\}$, and $j\in\{1,\dots,n_u\}$. Then
\begin{equation}\label{eq:distXwiYwj} 
\dist(X_{wi},Y_{uj})\geq \tfrac1{12}(\rho_w\e_{wi}\diam{U_{wi}}+\rho_u\e_{uj}\diam{U_{uj}}).
\end{equation}
Moreover, if $wi \neq uj$, then
\begin{equation}\label{eq:distXwiXwj} 
\min\{\dist(X_{wi},X_{uj}),\dist(Y_{wi},Y_{uj})\}  \geq  \tfrac1{12}(\rho_w\e_{wi}\diam{U_{wi}}+\rho_u\e_{uj}\diam{U_{uj}}).
\end{equation}
\end{Lem}

\begin{proof}
We distinguish two cases. If $w\neq u$, then by Lemma \ref{lem:distofhats} and the fact that $\rho_{w},\rho_{u} \in (0,1)$,
\begin{align*}
\dist(X_{wi},X_{uj}) &\geq d(x_{wi},x_{uj}) - \tfrac1{12}\left( \e_{wi}\rho_{w}\diam{U_{wi}} + \e_{uj}\rho_{u}\diam{U_{uj}}\right)\\
&\geq \dist(\dee U_{wi},\dee U_{uj}) - \tfrac1{12}\left( \e_{wi}\rho_{w}\diam{U_{wi}} + \e_{uj}\rho_{u}\diam{U_{uj}}\right)\\
&\geq \tfrac12 (\e_{wi}\diam{U_{wi}} + \e_{uj}\diam{U_{uj}}).
\end{align*}
The same argument works for $\dist(Y_{wi},Y_{uj})$ and for $\dist(X_{wi},Y_{uj})$.

If $w=u$, then recalling that $\e_{wi},\e_{wj} \in (0,1)$,
\begin{align*} 
\dist(X_{wi},X_{wj}) &\geq d(x_{wi},x_{wj}) - \tfrac1{12}\left( \e_{wi}\rho_{w}\diam{U_{wi}} + \e_{wj}\rho_{w}\diam{U_{wj}}\right)\\
&\geq \tfrac1{2}\left( \rho_{w}\diam{U_{wi}} + \rho_{w}\diam{U_{wj}}\right) - \tfrac1{12}\left( \e_{wi}\rho_{w}\diam{U_{wi}} + \e_{wj}\rho_{w}\diam{U_{wj}}\right)\\
&\geq \tfrac1{4}\left( \rho_{w}\diam{U_{wi}} + \rho_{w}\diam{U_{w}}\right).
\end{align*}
A similar argument works for $\dist(Y_{wi},Y_{uj})$ and for $\dist(X_{wi},Y_{uj})$.
\end{proof}

\subsection{Decomposition of the unit interval}

In the next lemma we define a collection of closed intervals $\{J_w : w\in\W\setminus\{\varepsilon\}\}$ in $[0,1]$ which form a defining sequence for some compact totally disconnected subset of $[0,1]$.
%which mirrors $\{U_w : w\in \W\}$. 
Recall our assumption that the only word in $\W$ of length one is 1.

\begin{Lem}\label{Lem:Decomposition}
There exists a family of closed intervals $\{J_w : w \in \W\setminus \{\varepsilon\}\}$ which satisfy the following properties.
\begin{enumerate}[(a)]
\item We have $J_1=[0,1]$.
\item For each $w \in \W\setminus \{\varepsilon\}$ and $i \in \{1,\dots,n_w\}$, $J_{wi} \subset J_w$ and $|J_{wi}| = (2n_w + 1)^{-1}|J_w|$.
\item For each $w \in \W\setminus \{\varepsilon\}$ and $i \in \{1,\dots,n_w-1\}$, $\dist(J_{wi},J_{w(i + 1)}) = (2n_w + 1)^{-1}|J_w|$ with the right endpoint of $J_{wi}$ lying to the left of the left endpoint of $J_{w(i+1)}$.
\item For each $w \in \W\setminus \{\varepsilon\}$ and $i \in \{1,n_w\}$, $\dist(J_{wi},\R \setminus J_w) = (2n_w + 1)^{-1}|J_w|$.
\end{enumerate}
\end{Lem}

\begin{proof}
Set $J_{1} = [0,1]$. Assume now that for some $w\in\W\setminus \{\varepsilon\}$ we have defined a nondegenerate closed interval $J_w = [a,b]$. For each $i\in \{1,\dots,n_w\}$ define 
\[ J_{wi} = \left[a + (2i-1)(2n_w+1)^{-1}|J_w|, a + 2i(2n_w+1)^{-1}|J_w|\right].\] 
A simple induction argument shows now that the intervals $\{J_w : w \in \W\}$ satisfy (a)--(d).
\end{proof}

For a given $w \in \W\setminus \{\varepsilon\}$, the set $J_w \setminus \bigcup_{i = 1}^{n_w} J_{wi}$ is a union of $n_w + 1$ open intervals. Denote the closures of these intervals by 
\[ \{I_{w,i} : i \in \{0,\dots,n_w\}\},\] 
ordered so that the right endpoint of $I_{w,i}$ lies to the left of the left endpoint of $I_{w,i+1}$ for each $i \in \{1,\dots,n_w - 1\}$.

\begin{Lem}\label{rem:I}
The collection $\{I_{w,i} : w \in \W\setminus \{\varepsilon\}, i \in \{0,\dots,n_w\}\}$ satisfies the following properties.
\begin{enumerate}[(a)]
\item For each $w \in \W\setminus \{\varepsilon\}$ and $i \in \{0,\dots,n_w\}$, $I_{w,i} \subset J_w$ and $|I_{w,i}| = (2n_w + 1)^{-1}|J_w|$.
\item For each $w \in \W\setminus \{\varepsilon\}$ and $i \in \{0,\dots,n_w-1\}$, $\dist(I_{w,i},I_{w,i+1}) = (2n_w + 1)^{-1}|J_w|$ with the right endpoint of $I_{w,i}$ lying to the left of the left endpoint of $I_{w,i+1}$.
\item If $I_{w,i}$ is adjacent to $I_{u,j}$, then either $w=u^{\uparrow}$ or $u=w^{\uparrow}$. In either case,
\[ |I_{w,i}| \geq (2n_w+1)^{-1}|I_{u,j}|.\]
%If $E$ is $c$-uniformly disconnected and if $I_{w,i}$ is adjacent to $I_{u,j}$, then 
%\[ |I_{w,i}| \geq (2N+1)^{-1}|I_{u,j}|.\]
\end{enumerate}
\end{Lem}

\begin{proof}
%The first two claims are immediate. For the third, assume without loss of generality that $|w|\geq |u|$. Then $I_{u,j}$ is adjacent to either one or two $J_{ul}$'s. Therefore, either $I_{w,i} = I_{u(j-1),n_{u(j-1)}}$ or $I_{w,i} = I_{uj,0}$. Assuming without loss of generality the latter,
%\[|I_{u,j}| = |J_{uj}| \geq |I_{w,i}| = (2n_{uj}+1)^{-1} |J_{uj}| \geq (2N+1)^{-1}|J_{uj}| = (2N+1)^{-1}|I_{u,j}|. \qedhere\]
The first two claims are immediate. For the third, assume first that $|w|\leq |u|$. By design, $w\neq u$. Therefore, $I_{w,i}$ is adjacent to $J_u$. But then, $w=u^{\uparrow}$. In this case,
\[ |I_{w,i}| = |J_u| \geq |I_{u,j}|. \]
Suppose now that $|w|\geq |u|$. Working as in the previous case, $u=w^{\uparrow}$ and
\[ |I_{w,i}| = (2n_w+1)^{-1}|J_w| = |I_{u,j}|. \qedhere\]
\end{proof}

For each $w \in \W\setminus \{\varepsilon\}$ and $i \in \{0,\dots,n_w\}$, denote the middle third of $I_{w,i}$ by $\hat I_{w,i}$ and set
\begin{equation}\label{eq:Lambda}
\Lambda_w = \frac{\diam{U_w}}{|I_{w,1}|} = (2n_w + 1)\frac{\diam{U_w}}{|J_w|}.
\end{equation}
Note that for all $i \in \{0,\dots,n_w\}$ we have $|I_{w,i}| = \Lambda_w^{-1}\diam U_w$ and $|\hat I_{w,i}| = (3\Lambda_w)^{-1}\diam U_w$.

\begin{Rem}
The limit set 
\begin{equation}\label{eq:J}
J_{\infty} = \bigcap_{n=1}^{\infty}\bigcup_{\substack{w\in\W \\ |w|=n}}J_w \subset J_1
\end{equation}
is compact and totally disconnected. Moreover, if $E$ is $c$-uniformly disconnected, then $J_{\infty}$ is $c'$-uniformly disconnected for some $c'(c)>1$ and the collection of intervals 
\[ \{I_{w,i} : w\in \W\setminus \{\varepsilon\}, \, i=0,\dots,n_w\}\] 
is a \emph{Whitney decomposition} of $J_1\setminus J_{\infty}$; see \cite[Theorem VI.1.1,Proposition VI.1.1]{Stein:1970}.
\end{Rem}

\section{Bi-Lipschitz arcs in \texorpdfstring{$X$}{X}}\label{Sec:Quasisymmetric2}

The crux of the proofs of Theorem \ref{thm:main1} and Theorem \ref{Thm:main2} is the the following proposition which is the main result of this section. 

\begin{Prop}\label{Prop:AllGen}
For each $w\in \W$ there exists a bi-Lipschitz embedding
\[ h_w : \bigcup_{i=1}^{n_w} \bigcup_{j=0}^{n_{wi}} \hat I_{wi,j} \to \bigcup_{i=1}^{n_w}\mathcal{N}_{wi} \]
with the following properties.
\begin{enumerate}
\item For each $w\in \W$ and each $i \in \{1,\dots,n_w\}$, 
\begin{enumerate}
\item the arc $h_w|\hat I_{wi,0}$ joins $X_{wi}$ to $X_{wi1}$ in $\mathcal{N}_{wi}$,
\item the arc $h_w|\hat I_{wi,j}$ joins $Y_{wij}$ to $X_{wi(j + 1)}$ in $\mathcal{N}_{wi}$ for $j \in \{1,\dots,n_{wi} - 1\}$,
\item the arc $h_w|\hat I_{wi,n_{wi}}$ joins $Y_{win_{wi}}$ to $Y_{wi}$ in $\mathcal{N}_{wi}$.
\end{enumerate}
\item For each $w\in\W$, $\dist(\Im(h_w),\bigcup_{u\in\W\setminus \{w\}}\Im(h_u)) > 0$.
\end{enumerate}

Moreover, if $E$ is $c$-uniformly disconnected, then, in addition to (1), there exists a constant $L(C_0,Q,C,c) \ge 1$ such that for each $w \in \W$, 
\begin{enumerate}
\item[(3)] $h_w$ is a $(L,\Lambda_w)$-quasisimilarity, and
\item[(4)] $\dist(\Im(h_w),\bigcup_{u\in\W\setminus \{w\}}\Im(h_u)) \geq L^{-1}\diam{U_w}$.
\end{enumerate} 
\end{Prop}

The construction of embeddings $\{h_w : w\in \W\}$ is done in two steps. In Section \ref{sec:even} we construct these embeddings for those words $w\in\W$ that have even number of letters, and in Section \ref{sec:odd} we construct them for those words $w\in\W$ that have odd number of letters; the extra complication is that curves defined in the second subsection must avoid those that have been defined in the first.
Recall that when we refer to $\texttt{data}$, we mean the constants $C_0$, $Q$, $C$, and $c$.

\subsection{Even generations}\label{sec:even}

Fix for the remainder of this subsection a word $w\in\W$ that has an even number of letters. The construction of the embedding $h_w$ is a double induction on the indices $i$ and $j$. For each $i \in \{1,\dots,n_w\}$, define
\[ \mathcal{I}_{wi} = \bigcup_{k=1}^i\bigcup_{j=0}^{n_{wk}}\hat I_{wk,j}.\]

\subsubsection{Base case} Here we construct a bi-Lipschitz embedding $h_{w,1}:\mathcal{I}_{w1} \to \mathcal{N}_{w1}$ in an inductive manner. For purposes of induction, set $Y_{w10} = X_{w1}$ and $X_{w1(n_{w1}+1)} = Y_{w1}$.
%For each $j_0 \in \{0,\dots, n_w\}$ set $\cal I_{w1}^{j_0} = \bigcup_{j = 0}^{j_0} \hat I_{w1}^j$.

By Proposition \ref{Prop:DefiningSequence}(D6), there is a curve $\sigma_{w1,0}$ in $\hat U_{w1}$ joining $Y_{w10}$ (i.e., $X_{w1}$) to $X_{w11}$. Using Lemma \ref{Lem:ApproximatingCurve} with $Y=\emptyset$ and $\lambda = \e_{w1}/8$, we obtain a bi-Lipschitz embedding $\gamma_{w1,0}:\hat I_{w1,0} \to X$ joining $Y_{w10}$ to $X_{w11}$ with
\begin{equation}\label{eq:Nw1A}
\sup_{z \in \gamma_{w1,0}}\dist(z,\sigma_{w1,0}) < \tfrac{1}{2}\e_{w1}\diam U_{w1}. 
\end{equation}
If $E$ is $c$-uniformly disconnected, then $\lambda = 2^{-9}\xi^{-3}$, and Lemma \ref{Lem:ApproximatingCurve} implies that $\g_{w1,0}:\hat I_{w1,0} \to X$ is a $(L_{1,0}',\Lambda_w)$-quasisimilarity for some constant $L_{1,0}'(\texttt{data}) \geq 1$. The curve $\sigma_{w1,0}$ lies in $\hat U_{w1}$, so by \eqref{eq:Nw1A}, the curve $\gamma_{w1,0}$ lies in $\mathcal{N}_{w1}$. Set now $h_{w,1,0} = \g_{w1,0}$.

For the inductive step, suppose that for some $j_0 \in \{0,\dots,n_{w1} - 1\}$, we have defined a bi-Lipschitz embedding $h_{w,1,j_0} : \bigcup_{j=0}^{j_0}\hat I_{w1,j} \to \mathcal{N}_{w1}$  such that the arc $h_{w,1,j_0}|\hat I_{w1,j}$ joins $Y_{w1j}$ to $X_{w1(j + 1)}$ for $j \in \{0,\dots,j_0\}$. Assume, moreover, that if $E$ is $c$-uniformly disconnected, then $h_{w,1,j_0}$ is a $(L_{1,j_0}',\Lambda_{w})$-quasisimilarity for some $L_{1,j_0}' \ge 1$.

By Proposition \ref{Prop:DefiningSequence}(D6), there exists a curve $\sigma_{w1,j_0 + 1}:[0,1] \to \hat U_{w1}$ joining $Y_{w1(j_0 + 1)}$ to $X_{w1(j_0 + 2)}$. Since $h_{w,1,j_0}$ is bi-Lipschitz, we have that $h_{w,1,j_0}(\hat I_{w1,0} \cup \cdots \cup \hat I_{w1,j_0})$ is $1$-HTB. Using Lemma \ref{Lem:ApproximatingCurve} with $\lambda = \e_{w1}/8$ and $Y = h_{w,1,j_0}(\hat I_{w1,0} \cup \cdots \cup \hat I_{w1,j_0})$ we obtain a bi-Lipschitz embedding $\gamma_{w1,j_0+1}:\hat I_{w1,j_0+1} \to X$ joining $Y_{w1(j_0 + 1)}$ to $X_{w1(j_0 + 2)}$ disjoint from $Y$ such that
\begin{equation*}
\sup_{z \in \gamma_{w1,j_0+1}}\dist(z,\sigma_{w1,j_0+1}) < \tfrac12\e_{w1} \diam U_{w1}.
\end{equation*}
If $E$ is $c$-uniformly disconnected, then $\lambda=2^{-9}\xi^{-3}$ and $Y$ is $(C_{1,j_0},1)$-HTB for some constant $C_{1,j_0}(\texttt{data},L_{1,j_0}')\ge 1$. Hence, $\gamma_{w1,j_0+1}$ is an $(L_{1,j_0 + 1}',\Lambda_w)$-quasisimilarity for some constant $L_{1,j_0 + 1}'(\texttt{data},L_{1,j_0}') \ge 1$ and there exists $\beta_{1,j_0 + 1}(\texttt{data},L_{1,j_0}') \ge 1$ such that
\begin{align}
\inf_{z \in \gamma_{w1,j_0 + 1}} \dist(z,Y) & \ge 2^{-9}\beta_{1,j_0 + 1}\xi^{-3}\diam(\sigma_{w1,j_0 + 1}) \label{eq:beta} \\
& \ge 2^{-9}\beta_{1,j_0 + 1}\xi^{-3}\dist(Y_{w1(j_0 + 1)},X_{w1(j_0 + 2)}) \nonumber \\
& \ge 2^{-19}\beta_{1,j_0 + 1}\xi^{-6}\rho(\diam U_{w1(j_0 + 1)} + \diam U_{w1(j_0 + 2)}) \nonumber \\
& \ge 2^{-20}\beta_{1,j_0 + 1}\xi^{-7}\rho\diam U_{w1} \nonumber \\
& =: \eta_{1,j_0 + 1}\diam U_{w1} \nonumber
\end{align}
by \eqref{eq:distXwiYwj} and \eqref{eq:diamUwi} (recalling that $\delta = (4\xi)^{-1}$, $\eps_{w1} = (4\xi)^{-3}$, and that $\rho_w = \rho$ is independent of $w$ by Remark \ref{rem:rho}).
As in the case $j = 0$, we have that $\gamma_{w1,j_0 + 1}$ lies in $\mathcal{N}_{w1}$. The next lemma completes the inductive step.

\begin{Lem}\label{Lem:Union}
Let $h_{w,1,j_0+1} : \hat I_{w1,0} \cup \cdots \cup \hat I_{w1,j_0+1} \to \mathcal{N}_{w1}$ be defined by 
\[ h_{w,1,j_0+1}|(\hat I_{w1,0} \cup \cdots \cup \hat I_{w1,j_0}) = h_{w,1,j_0} \quad\text{and}\quad h_{w,1,j_0+1}|\hat I_{w1,j_0+1} = \gamma_{w1,j_0+1}.\] Then $h_{w,1,j_0+1}$ is bi-Lipschitz. Moreover, if $E$ is $c$-uniformly disconnected, then $h_{w,1,j_0+1}$ is an $(L_{1,j_0 + 1},\Lambda_w)$-quasisimilarity for some constant $L_{1,j_0 + 1}(\texttt{data},L_{1,j_0})\ge 1$.
\end{Lem}

\begin{proof}
Since $h_{w,1,j_0}$ and $\gamma_{w1,j_0+1}$ are bi-Lipschitz and have disjoint images, it follows that $h_{w,1,j_0+1}$ is bi-Lipschitz. 

Assume now that $E$ is $c$-uniformly disconnected. By the inductive step and by design $h_{w,1,j_0}$ and $\gamma_{w1,j_0+1}$ are quasisimilarities. Let now $s \in \hat I_{w1,0} \cup \cdots \cup \hat I_{w1,j_0}$ and $t \in \hat I_{w1,j_0+1}$. On the one hand,
\[ (2N+1)^{-2} \leq (2N+1)^{-1}\frac{|J_{w1}|}{|J_w|} \leq \frac{\dist(I_{w1,j_0}, I_{w1,j_0+1})}{|J_w|} \leq \frac{|s-t|}{|J_w|} \leq \frac{|J_{w1}|}{|J_w|} \leq 1, \]
which implies that
\[ (2N+1)^{-2}\Lambda_w^{-1}\diam U_w \leq |s-t| \leq \Lambda_w^{-1}\diam U_w.\]
On the other hand, by \eqref{eq:beta} and \eqref{eq:diamUwi},
\begin{align*}
(4\xi)^{-1}\eta_{1,j_0 + 1}\diam U_w  \le \eta_{1,j_0 + 1}\diam U_{w1} \leq d(h_{w1,j_0+1}(s),h_{w1,j_0+1}(t)) & \leq \diam{\mathcal{N}_{w1}} \\
& \leq \diam{U_w}. \qedhere
\end{align*}
\end{proof}

Proceeding inductively, we obtain a bi-Lipschitz embedding $h_{w,1} := h_{w,1,n_{w1}}$ such that property (1) of Proposition \ref{Prop:AllGen} holds for $i=1$. Moreover, if $E$ is $c$-uniformly disconnected, then $h_{w,1}$ is a $(L_1,\Lambda_w)$-quasisimilarity for some constant $L_1(\texttt{data}) \geq 1$.

\subsubsection{Inductive step}

Suppose that for some $i_0 \in \{1,\dots,n_w-1\}$, we have constructed a bi-Lipschitz embedding $h_{w,i_0}:\mathcal{I}_{wi_0}  \to \bigcup_{i=1}^{i_0}\mathcal{N}_{wi}$ such that for each $i\in \{1,\dots,i_0\}$, the arc $h_{w,i_0}|\hat I_{wi,0}$ satisfies property (1) of Proposition \ref{Prop:AllGen}. Moreover, suppose that if $E$ is $c$-uniformly disconnected, then $h_{w,i_0}$ is an $(L_{i_0},\Lambda_w)$-quasisimilarity for some constant $L_{i_0} \geq 1$.

As with $h_{w,1}$, the construction of $h_{w,i_0+1}: \mathcal{I}_{w(i_0+1)} \to \bigcup_{i=1}^{i_0+1}\mathcal{N}_{wi}$ is in an inductive fashion and we only outline the steps. Set $Y_{w(i_0+1)0} = X_{w(i_0+1)}$ and $X_{w(i_0+1)(n_{w(i_0+1)}+1)} = Y_{w(i_0+1)}$. By Proposition \ref{Prop:DefiningSequence}(D6), we find a curve $\sigma_{w(i_0 + 1),0}$ joining $Y_{w(i_0+1)0}$ to $X_{w(i_0 + 1)1}$ in $\hat{U}_{w(i_0 + 1)}$. Since $h_{w,i_0}$ is bi-Lipschitz, its image $Y$ is 1-HTB and applying Lemma \ref{Lem:ApproximatingCurve} with $\lambda = \e_{w(i_0+1)}/8$, we obtain a bi-Lipschitz curve $\gamma_{w(i_0+1),0} : \hat I_{w(i_0+1),0} \to \mathcal{N}_{w(i_0+1)}$ joining $Y_{w(i_0+1)0}$ to $X_{w(i_0+1)1}$ whose image is disjoint from $Y$. If $E$ is $c$-uniformly disconnected, then the image of $h_{w,i_0}$ is $(C_{i_0},1)$-HTB, which implies that $\gamma_{w(i_0+1),0}$ is an $(L_{i_0+1,0}',\Lambda_w)$-quasisimilarity. Now define $h_{w,i_0+1,0} : \mathcal{I}_{wi_0} \cup \hat I_{w(i_0+1),0} \to \bigcup_{i=1}^{i_0+1}\mathcal{N}_{wi}$ by
\[ h_{w,i_0+1,0} | \mathcal{I}_{wi_0} = h_{w,i_0}\quad \text{and}\quad h_{w,i_0+1,0} | \hat I_{w(i_0+1),0} = \gamma_{w(i_0+1),0}.\] 
Working as in Lemma \ref{Lem:Union}, we can show that $h_{w,i_0+1,0}$ is bi-Lipschitz, and if $E$ is $c$-uniformly disconnected, then $h_{w,i_0+1,0}$ is an $(L_{i_0+1,0},\Lambda_w)$-quasisimilarity.

Assuming we have defined $h_{w,i_0+1,j_0}$ for some $j_0 \in \{0,\dots,n_{w(i_0+1)}-1\}$, we use Proposition \ref{Prop:DefiningSequence}(D6) to find a curve $\sigma_{w(i_0 + 1),j_0+1}$ joining $Y_{w(i_0 + 1)(j_0+1)}$ to $Y_{w(i_0 + 1)(j_0+2)}$ in $\hat{U}_{w(i_0+1)}$. The image $Y$ of $h_{w,i_0 + 1,j_0}$ is 1-HTB since $h_{w,i_0 + 1,j_0}$ is bi-Lipschitz, so by Lemma \ref{Lem:ApproximatingCurve} (with $\lambda = \e_{w(i_0+1)}/8$), we obtain a bi-Lipschitz embedding $\gamma_{w(i_0+1),j_0+1} : \hat I_{w(i_0+1),j_0+1} \to \mathcal{N}_{w(i_0+1)}$ which is disjoint from $Y$. Moreover, if $E$ is $c$-uniformly disconnected, then $\gamma_{w(i_0+1),j_0+1}$ is an $(L_{i_0+1,j_0+1}',\Lambda_w)$-quasisimilarity. Working as in Lemma \ref{Lem:Union}, we get that the map
\[ h_{w,i_0+1,j_0+1} : \mathcal{I}_{wi_0} \cup \bigcup_{j=0}^{j_0+1}\hat I_{w(i_0+1),j} \to  \bigcup_{i=1}^{i_0+1}\mathcal{N}_{wi}\]
defined by
\[ h_{w,i_0+1,j_0+1}|\mathcal{I}_{wi_0} \cup \bigcup_{j=0}^{j_0}\hat I_{w(i_0+1),j} = h_{w,i_0+1,j_0} \quad \text{and}\quad h_{w,i_0+1,j_0+1}|\hat I_{w(i_0+1),j_0+1} = \gamma_{w(i_0+1),j_0+1}\]
is bi-Lipschitz and, if $E$ is $c$-uniformly disconnected, it is an $(L_{i_0+1,j_0+1},\Lambda_w)$-quasisimilarity.

Proceeding inductively, we obtain a bi-Lipschitz embedding 
\[ h_{w,i_0+1} = h_{w,i_0+1,n_{w(i_0+1)}} : \mathcal{I}_{w(i_0+1)} \to \bigcup_{i=1}^{i_0+1} \mathcal{N}_{wi}\] 
such that property (1) of Proposition \ref{Prop:AllGen} holds for $i\in\{1,\dots,i_0+1\}$.

\subsubsection{End of induction}
After $n_w$ many steps we obtain a bi-Lipschitz embedding 
\[ h_w := h_{w,n_w} : \mathcal{I}_{wn_w} \to \bigcup_{i=1}^{n_w}\mathcal{N}_{wi} \]
that satisfies property (1) of Proposition \ref{Prop:AllGen}. Moreover, if $E$ is $c$-uniformly disconnected, then $n_w \leq N(\texttt{data})$ and $h_w$ is a $(L',\Lambda_w)$-quasisimilarity for some $L'(\texttt{data})\geq 1$.

\begin{figure}[ht]
\centering
\includegraphics[width=\textwidth]{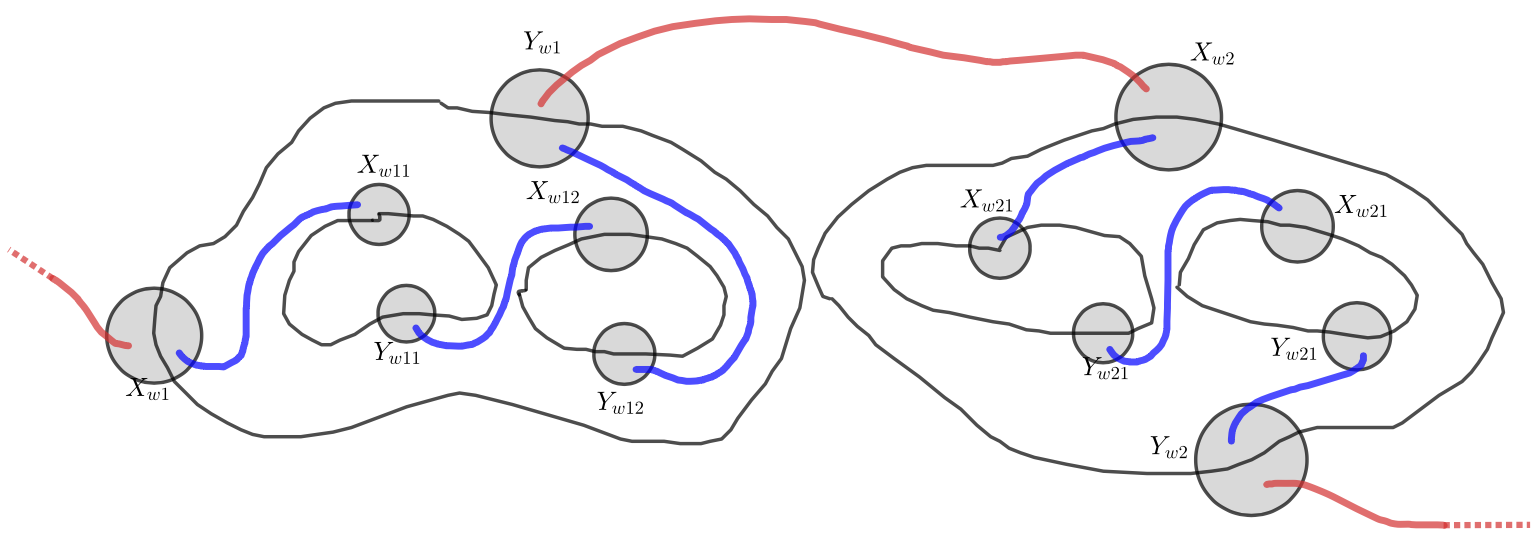}
\caption{In this example, $w$ has an even number of letters, $n_w = 2$ and $n_{w1}=n_{w2}=2$. The blue curves are the image of $h_{w}$ defined in \textsection\ref{sec:even} while the red curves are part of the image of $h_{w^{\uparrow}}$ defined in \textsection\ref{sec:odd}.}
\label{fig:curves}
\end{figure}

\subsection{Odd generations}\label{sec:odd}

Fix for this subsection a word $w \in \W$ that has an odd number of letters.
By \textsection\ref{sec:even}, we have constructed bi-Lipschitz embeddings $h_{w^{\uparrow}},h_{w1},\dots,h_{wn_w}$ such that, if $E$ is $c$-uniformly disconnected, then $h_w$ is an $(L',\Lambda_{w^{\uparrow}})$-quasisimilarity and each $h_{wi}$ is an $(L',\Lambda_{wi})$-quasisimilarity. The construction of the embedding $h_w$ mirrors the previous subsection, but here we also avoid the images of $h_{w^{\uparrow}},h_{w1},\dots,h_{wn_w}$.

Denote by $H$ the union of the images of $h_{w^{\uparrow}},h_{w1},\dots,h_{wn_w}$,. 
Since these maps are bi-Lipschitz, $H$ is 1-HTB and, by Lemma \ref{lem:BHR}, $H$ is porous. Further, if $E$ is $c$-uniformly disconnected, then $H$ is $(C',1)$-HTB for some $C'(\texttt{data}) \ge 1$ and $H$ is $p_0$-porous for some $p_0(\texttt{data}) \ge 1$.
By porosity, there exist balls $X_{wi}' \subset X_{wi}$ and $Y_{wi}' \subset Y_{wi}$ which are disjoint from $H$; if $E$ is $c$-uniformly disconnected, these balls can be taken to have diameters $(2p_0)^{-1}\diam X_{wi}$ and $(2p_0)^{-1}\diam Y_{wi}$, respectively and have distance from $H$ at least $(2p_0)^{-1}\diam X_{wi}$ and $(2p_0)^{-1}\diam Y_{wi}$, respectively.
Similarly, for each $i \in \{1,\ldots,n_w\}$ and $j \in \{1,\ldots,n_{wi}\}$, there exist balls $X_{wij}' \subset X_{wij}$ and $Y_{wij}' \subset Y_{wij}$ which are disjoint from $H$ and if $E$ is $c$-uniformly disconnected, the diameters of these balls are $(2p_0)^{-1}\diam X_{wij}$ and $(2p_0)^{-1}\diam Y_{wij}$, respectively, and they have distance from $H$ at least $(2p_0)^{-1}\diam X_{wij}$ and $(2p_0)^{-1}\diam Y_{wij}$, respectively.

We give a sketch of how to define $h_w$.
We apply the same methods as in \textsection\ref{sec:even}, but now also avoid $H$.
Specifically, the only difference is when applying Lemma \ref{Lem:ApproximatingCurve}; instead of taking $Y$ to be only the union $H'$ of the images of the bi-Lipschitz embedding constructed prior to the inductive step, we take $Y = H \cup H'$, which allows the new bi-Lipschitz curve $\gamma$ to be disjoint from $H$ as well.
Indeed, $H \cup H'$ is 1-HTB, as the union of two 1-HTB sets.

If $E$ is $c$-uniformly disconnected, then $H \cup H'$ is $(C'',1)$-HTB for some $C''(\texttt{data}) \ge 1$ and the curve $\gamma$ is an $(L'',\Lambda_w)$-quasisimilarity for some $L''(\texttt{data}) \ge 1$.
We proceed by extending the definition of $h_w$ using the curve $\gamma$, as before.

In the end, after finitely many steps, we obtain a bi-Lipschitz embedding 
$h_w$ and, provided $E$ is $c$-uniformly disconnected, $h_w$ is an $(L,\Lambda_w)$-quasisimilarity for some $L(\texttt{data}) \ge 1$ and there exists a constant $\eta(\texttt{data}) \in (0,1)$ such that
\begin{equation}\label{eq:betaIneq}
\dist\left(\Im(h_w),\Im(h_{w^{\uparrow}}) \cup \bigcup_{i = 1}^{n_w} \Im(h_{wi})\right) \ge \eta\diam U_w.
\end{equation}

\subsection{Proof of Proposition \ref{Prop:AllGen}}

It is immediate from design that embeddings $\{h_w : w\in \W\}$ satisfy properties (1) and (3) of Proposition \ref{Prop:AllGen}.

Fix now distinct $w,u \in \W$. Suppose first that one is not the parent of the other. Fix also a point $p \in \Im(h_w)$ and a point $q \in \Im(h_u)$. Then $p \in \mathcal{N}_{wi}$ and $q \in \mathcal{N}_{uj}$ for some $i \in \{1,\ldots,n_w\}$ and $j \in \{1,\ldots,n_u\}$. By Lemma \ref{lem:distofhats},
\[ d(p,q) \ge \dist(\mathcal{N}_{wi},\mathcal{N}_{uj}) \ge \epsilon_{wi}\diam U_{wi} + \epsilon_{uj}\diam U_{uj}, \]
which gives properties (2) and (4).

Suppose now that $u=wi$ for some $i\in\{1,\dots,n_w\}$. By design (see \textsection\ref{sec:odd}), the images of $h_w$ and $h_{wi}$ are disjoint, which gives property (2). For property (4), suppose $E$ is $c$-uniformly disconnected.
If $u$ has an odd number of letters, then $w$ has an even number of letters, while if $u$ has an even number of letters, then $w$ has an odd number of letters. In either case, by \eqref{eq:betaIneq},
%\[ \dist(\Im(h_u),\Im(h_w)) \ge \eta\min\{\diam U_u, \diam U_w\}. \]
%Instead if $u$ has an even number of letters, by \eqref{eq:betaIneq},
%\[ \dist(\Im(h_w),\Im(h_u)) \ge \eta\diam U_w. \]
%In either case,
\[ d(p,q) \ge \dist(\Im(h_w),\Im(h_u)) \ge \eta\min\{\diam U_w,\diam U_u\}. \]

\section{Local modifications}\label{sec:modif}

In this section, within each ball $X_w$ and $Y_w$, with $|w|\geq 2$, we patch curves defined in Proposition \ref{Prop:AllGen} together with a bi-Lipschitz curve using Lemma \ref{Lem:ApproximatingCurve}, then delete overlapping portions and join with geodesics. These local modifications are identical to those in \cite[Section 6.1]{HVZ24}, aside from notational changes. 

Here and for the rest of this section, given a non-degenerate interval $[a,b]$ and two points $p,q\in X$, we say that a curve $g: [a,b] \to X$ is \emph{a geodesic joining $p$ to $q$} if $g(a) = p$, $g(b)=q$ and for all $t,s \in [a,b]$ with $t < s$ we have that
\[ \frac{d(g(s),g(t))}{s - t} = \frac{d(g(a),g(b))}{b-a}.\]
The constant on the right-hand side is called the \emph{scaling factor} of the geodesic $g$.

Recall the definition of the totally disconnected set $J_{\infty} \subset [0,1]$ from \eqref{eq:J}. Given our assumption that the only word with length one is 1, we have that $J_{\infty} \subset J_1$. It is clear from Lemma \ref{Lem:Decomposition} and Lemma \ref{rem:I} that the collection $\{I_{w,i} : w \in \W\setminus \{\varepsilon\}, i \in \{0,\ldots,n_w\}\}$ is a collection of closed subintervals of $J_1\setminus J_{\infty}$ with disjoint interiors such that
\[ J_1 = J_{\infty} \cup \bigcap_{w\in \W\setminus \{\varepsilon\}} \bigcup_{i=0}^{n_w} I_{w,i}. \]
Write $I_{w,i} = [a_{w,i},b_{w,i}]$ and $\hat{I}_{w,i} = [\hat{a}_{w,i},\hat{b}_{w,i}]$.
Set
\begin{equation}\label{eq:kappa}
\kappa_w = (\diam U_w)^{-1}\min\Bigg\{\dist\left(\Im(h_w),\bigcup_{v \in \W \setminus \{w\}} \Im(h_v)\right), \min_{\substack{i \in \{1,\ldots,n_w\} \\ j \in \{0,\ldots,n_{wi}\}}} d(h_w(\hat a_{wi,j}),h_w(\hat b_{wi,j}))\Bigg\},
\end{equation}
which is positive by Proposition \ref{Prop:AllGen}.
If $E$ is $c$-uniformly disconnected, then the right hand side of \eqref{eq:kappa} is bounded from below by $(3L(2N + 1))^{-1}$ (where $N$ is the constant of Proposition \ref{Prop:DefiningSequence} and $L$ is the constant of Proposition \ref{Prop:AllGen}) and we can set 
\[ \kappa_w = (3L(2N + 1))^{-1}.\]

With the exception of $I_{1,0}$ and $I_{1,n_1}$, each interval $I_{w,i}$ has an adjacent interval to its left and an adjacent interval to its right. In particular, given $w\in \W$, $i\in \{1,\dots,n_w\}$, and $j\in \{0,\dots,n_w\}$,
\begin{itemize}
\item if $j\neq n_{w}$, then the interval $I_{wi1,0}$ is to the right of $I_{wi,j}$, while if $j = n_w$ and $wi \neq 1$, then the interval $I_{w,i + 1}$ is to the right of $I_{wi,j}$;
%left with $I_{w,i-1}$ and to its right with $I_{wi1,0}$,
\item if $j \neq 0$, then the interval $I_{wij,n_{wij}}$ is to the left of $I_{wi,j}$, while if $j = 0$ and $wi \neq 1$, then the interval $I_{w,i - 1}$ is to the left of $I_{wi,j}$;
%$j=n_w$ and $wi\neq 1$, then $I_{wi,j}$ is adjacent to its left with $I_{wij,n_{wij}}$ and to its right with $I_{w,i+1}$,
%\item if $j \neq 0, n_w$, then $I_{wi,j}$ is adjacent to its left with $I_{wij,n_{wij}}$ and to its right with $I_{wi1,0}$,
\item $I_{1,0}$ has no interval to its left and $I_{1,n_1}$ has no interval to its right.
\end{itemize}

Let $\mathscr{I} = \{I_{w,i} : w \in \W, i \in \{0,\ldots,n_w\}\}$, and let $\mathscr{I} = \mathscr{I}_1 \cup \mathscr{I}_2$ be a partition of $\mathscr{I}$ into two disjoint subsets such that no two intervals in $\mathscr{I}_1$ are adjacent and no two intervals of $\mathscr{I}_2$ are adjacent.
We may assume that $I_{1,0}$ and $I_{1,n_1}$ are intervals of $\mathscr{I}_1$.
Note that each interval of $\mathscr{I}_2$ lies between two intervals of $\mathscr{I}_1$, and each interval of $\mathscr{I}_1 \setminus \{I_{1,0},I_{1,n_1}\}$ lies between two intervals of $\mathscr{I}_2$.
In the next subsection, we perform local modifications on intervals of $\mathscr{I}_1$, and then on intervals of $\mathscr{I}_2$ in the section following that.
The only difference is that we take care to consider the local modifications done on intervals of $\mathscr{I}_1$ when performing them on intervals of $\mathscr{I}_2$.
Recall that $\texttt{data}$ refers to the constants $C_0$, $Q$, $C$, and $c$, and recall the constant $L$ from Proposition \ref{Prop:AllGen}.

\subsection{Local modifications on intervals of \texorpdfstring{$\mathscr{I}_1$}{I1}}\label{sec:modif1}

Let $I_{wi,j} \in \mathscr{I}_1 \setminus \{I_{1,0}\}$.
We assume that $I_{ui',j'}$ is the interval to the left of $I_{wi,j}$; see Figure \ref{fig:ab}.
The points $h_u(\hat b_{ui',j'})$ and $h_w(\hat a_{wi,j})$ lie in a ball $B_{wij} \in \{X_{wij},Y_{wij}\}$ by Proposition \ref{Prop:AllGen}(1).
By Lemma \ref{rem:I}, the embedding $h_u$, in addition to being an $(L,\Lambda_u)$-quasisimilarity, is a $(\bar L,\Lambda_w)$-quasisimilarity for some $\bar L(L,\texttt{data}) \ge 1$ since one of $u$ or $w$ is the parent of the other.
%We define bi-Lipschitz embeddings of $[\hat b_{ui',j'},\hat a_{wi,j}]$ into a small neighborhood of $B_{wij}$, then modify the resulting concatenation to ensure it is locally bi-Lipschitz.

\begin{figure}[ht]
\centering
\includegraphics[width=0.6\textwidth]{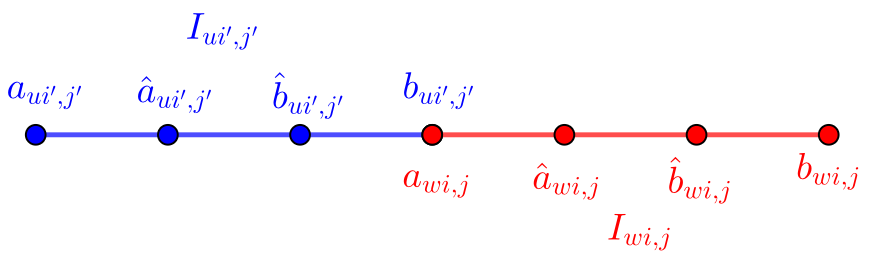}
\caption{The positions of $\hat a_{ui',j'}$, $\hat b_{ui',j'}$, $\hat a_{wi,j}$, and $\hat b_{wi,j}$. Note that $b_{ui',j'} = a_{wi,j}$.}
\label{fig:ab}
\end{figure}

Set
\[ H_{(ui',j'),(wi,j)} = \left(\Im(h_{w^{\uparrow}}) \cup \Im(h_w) \cup \bigcup_{i = 1}^{n_w} \Im(h_{wi})\right) \setminus (h_u(\hat I_{ui',j'}) \cup h_w(\hat I_{wi,j})). \]
The ball $B_{wij}$ is path connected, so there is a curve $\sigma_{(ui',j'),(wi,j)}:[0,1] \to B_{wij}$ joining $h_u(\hat b_{ui',j'})$ to $h_w(\hat a_{wi,j})$.
Using Lemma \ref{Lem:ApproximatingCurve} with $\lambda = \frac{1}{8}$ and $Y = H_{(ui',j'),(wi,j)}$, we find a bi-Lipschitz arc
\[ \Gamma_{(ui',j'),(wi,j)}:[\hat b_{ui',j'},\hat a_{wi,j}] \to X \]
such that $\Gamma_{(ui',j'),(wi,j)}(\hat b_{ui',j'}) = h_u(\hat b_{ui',j'})$, $\Gamma_{(ui',j'),(wi,j)}(\hat a_{wi,j}) = h_w(\hat a_{wi,j})$,
\begin{equation}\label{eq:close2}
\dist(\Gamma_{(ui',j'),(wi,j)},H_{(ui',j'),(wi,j)}) > 0
\end{equation}
and, for all $z \in \Gamma_{(ui',j'),(wi,j)}$,
\begin{equation}\label{eq:close1}
\dist(z,\sigma_{(ui',j'),(wi,j)}) \le \tfrac{1}{2}\diam(\sigma_{(ui',j'),(wi,j)}) \le \tfrac{1}{2}\diam B_{wij} \le \frac{1}{24}\epsilon_{wij}\rho_{wi}\diam U_{wij}.
\end{equation}
Furthermore, if $E$ is $c$-uniformly disconnected, then there exist $L^*(\texttt{data}) \ge 1$ and $\beta^*(\texttt{data}) \in (0,1)$ such that $\Gamma_{(ui',j'),(wi,j)}$ is an $(L^*,\Lambda_w)$-quasisimilarity and, by Proposition \ref{Prop:AllGen}(4), for each $z \in \Gamma_{(ui',j'),(wi,j)}$,
\begin{align}\label{eq:beta*}
\dist(z,H_{(ui',j'),(wi,j)}) & \ge \beta^*\diam(\sigma_{(ui',j'),(wi,j)}) \ge \beta^*d(h_u(\hat b_{ui',j'}),h_w(\hat a_{wi,j})) \geq \frac{\beta^*}{L}\diam U_w.
\end{align}

Inequality \eqref{eq:close1} ensures that the arc $\Gamma_{(ui',j'),(wi,j)}$ does not exit the ball $\tfrac32B_{wij}$, hence it does not leave $\cal N_{wi}$ and therefore cannot intersect with any arc $\Im(h_v)$ for any $v \in \W$ with $||v| - |w|| > 1$.
Inequality \eqref{eq:close2} ensures that $\Gamma_{(ui',j'),(wi,j)}$ does not intersect with any arc $\Im(h_v)$ such that $v \in \W$ with $||v| - |w|| \le 1$, aside from possibly in the arcs $h_u(\hat I_{ui',j'})$ and $h_w(\hat I_{wi,j})$. In the case where $E$ is $c$-uniformly disconnected, by \eqref{eq:beta*}, $\Gamma_{(ui',j'),(wi,j)}$ has quantitative distance from any arc $\Im(h_v)$ such that $v \in \W$ with $||v| - |w|| \le 1$, aside from the arcs $h_u(\hat I_{ui',j'})$ and $h_w(\hat I_{wi,j})$.

We now define a few points in $[\hat a_{ui',j'},\hat a_{wi,j}]$.
Set 
\[ \theta_{wi,j} = (2L)^{-1}\beta^*\epsilon_w\epsilon_{wi}\epsilon_{wij}\rho_{wi}. \]

\begin{Rem}
If $E$ is uniformly disconnected, since the quantities $\epsilon_w, \epsilon_{wi}, \epsilon_{wij}$ and $\rho_{wi}$ are chosen independent of $(wi,j)$, so is $\theta_{wi,j}$ chosen independent of $(wi,j)$ and we simply denote it by $\theta$.  
\end{Rem}

Now define
\begin{align*}
\til b_{ui',j'} &= \min\{t \in [\hat a_{ui',j'},\hat b_{ui',j'}] : \dist(h_u(t),\Gamma_{(ui',j'),(wi,j)}([\hat b_{ui',j'},\hat a_{wi,j}])) = \theta_{wi,j}\diam U_w\},\\
t_{ui',j'}^R &= \max\{t \in [\hat b_{ui',j'},\hat a_{wi,j}] : \dist(h_u(\til b_{ui',j'}),\Gamma_{(ui',j'),(wi,j)}(t)) = \theta_{wi,j}\diam U_w\}.
\end{align*}
Note that
\begin{equation*}
d(h_u(\til b_{ui',j'}),h_u(\hat b_{ui',j'})) \ge \dist(h_u(\til b_{ui',j'}),\Gamma_{(ui',j'),(wi,j)}([\hat b_{ui',j'},\hat a_{wi,j}])) \ge \theta_{wi,j}\diam U_w,
\end{equation*}
so $\hat b_{ui',j'} - \til b_{ui',j'} >0$ and, if $E$ is $c$-uniformly disconnected, then
\begin{equation}\label{eq:bt}
\hat b_{ui',j'} - \til b_{ui',j'} \ge (L\Lambda_w)^{-1}\theta\diam U_w.
\end{equation}
Moreover, by choice of $\theta_{wi,j}$,
\begin{align*}
d(&\Gamma_{(ui',j'),(wi,j)}(t_{ui',j'}^R),\Gamma_{(ui',j'),(wi,j)}(\hat a_{wi,j})) \\
& \ge \dist(h_w(\hat a_{wi,j}),h_u([\hat a_{ui',j'},\hat b_{ui',j'}])) - \dist(\Gamma_{(ui',j'),(wi,j)}(t_{ui',j'}^R),h_u([\hat a_{ui',j'},\hat b_{ui',j'}])) \\
& \ge \kappa_w\diam U_w - \theta_{wi,j}\diam U_w \\
& \ge \theta_{wi,j}\diam U_w,
\end{align*}
so $\hat a_{wi,j} - t_{ui',j'}^R >0$. If $E$ is $c$-uniformly disconnected, then
\begin{equation}\label{eq:at}
\hat a_{wi,j} - t_{ui',j'}^R \ge (L^*\Lambda_w)^{-1}\theta\diam U_w.
\end{equation}
Further, \eqref{eq:bt} and Lemma \ref{rem:I} imply that
\begin{align} \label{eq:t3t4}
\frac{\theta\diam U_w}{L\Lambda_w} & \le \hat b_{ui',j'} - \til b_{ui',j'} \le t_{ui',j'}^R - \til b_{ui',j'} \le 2(|\hat I_{ui',j'}| + |\hat I_{wi,j}|) 
%& \le 2|J_{wi}| \notag \\
\le \frac{(4N+2)\diam U_w}{\Lambda_w}. 
\end{align}

Fix a geodesic $g_{ui',j'}^R:[\til b_{ui',j'},t_{ui',j'}^R] \to X$ joining $h_u(\til b_{ui',j'})$ to $\Gamma_{(ui',j'),(wi,j)}(t_{ui',j'}^R)$ and note that by choice of $\theta_{wi,j}$, $g_{ui',j'}^R$ does not leave the ball $2B_{wij}$, and by extension, does not exit $\cal N_{wi}$.
The superscript ``$R$'' here is to denote that this geodesic and the point $t^R_{ui',j'}$ are defined on the right side of the interval $I_{ui',j'}$.
By \eqref{eq:t3t4}, the scaling factor of $g_{ui',j'}^R$ satisfies
\begin{equation}\label{eq:similarity}
\frac{\theta_{wi,j}\Lambda_w}{4N + 2} \le \frac{d(h_u(\til b_{ui',j'}),\Gamma_{(ui',j'),(wi,j)}(t_{ui',j'}^R))}{t_{ui',j'}^R - \til b_{ui',j'}} \le L\Lambda_w.
\end{equation}
Now set
\begin{align*}
\til a_{wi,j} &= \max\{t \in [\hat a_{wi,j},\hat b_{wi,j}] : \dist(h_w(t),\Gamma_{(ui',j'),(wi,j)}([\hat b_{ui',j'},\hat a_{wi,j}])) = \theta_{wi,j}\diam U_w\},\\
t_{wi,j}^L &= \min\{t \in [t_{ui',j'}^R,\hat a_{wi,j}] : d(\Gamma_{(ui',j'),(wi,j)}(t),h_w(t_{wi,j}^R)) = \theta_{wi,j}\diam U_w\}.
\end{align*}
%See Figure \ref{fig:ts} for a visual.

\begin{figure}[ht]
\centering
\includegraphics[width=0.65\textwidth]{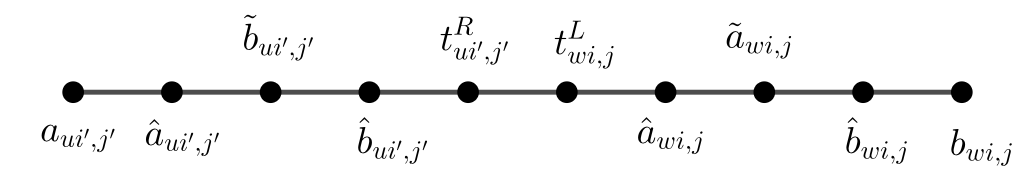}
\caption{The positions of $\til b_{ui',j'}$, $t_{ui',j'}^R$, $t_{wi,j}^L$, and $\til a_{wi,j}$.}
\label{fig:ts}
\end{figure}

As with \eqref{eq:at}, $\hat a_{wi,j} - t_{wi,j}^L > 0$, and, if $E$ is $c$-uniformly disconnected,
\begin{equation}\label{eq:at2}
\hat a_{wi,j} - t_{wi,j}^L \ge (L^*\Lambda_w)^{-1}\theta\diam U_w.
\end{equation}
Moreover, for any $t \in [\hat b_{ui',j'},\hat a_{wi,j}]$ satisfying $d(\Gamma_{(ui',j'),(wi,j)}(t),h_w(t_{wi,j}^R)) = \theta_{wi,j}\diam U_w$,
\begin{align*}
d(& \Gamma_{(ui',j'),(wi,j)}(t_{ui',j'}^R),\Gamma_{(ui',j'),(wi,j)}(t)) \\
& \ge \dist(h_u([\hat a_{ui',j'},\hat b_{ui',j'}]),h_w([\hat a_{wi,j},\hat b_{wi,j}])) - 2\theta_{wi,j}\diam U_w \\
& \ge \kappa_w\diam U_w - 2\theta_{wi,j}\diam U_w \\
& \ge \theta_{wi,j}\diam U_w
\end{align*}
by choice of $\theta_{wi,j}$.
Thus $t_{wi,j}^L$ is well-defined, $t_{wi,j}^L - t_{ui',j'}^R >0$ and, if $E$ is $c$-uniformly disconnected,
\begin{equation}\label{eq:tt}
t_{wi,j}^L - t_{ui',j'}^R \ge (L^*\Lambda_w)^{-1}\theta\diam U_w.
\end{equation}

Fix a geodesic $g_{wi,j}^L:[t_{wi,j}^L,\til a_{wi,j}] \to X$ joining $\Gamma_{(ui',j'),(wi,j)}(t_{wi,j}^L)$ to $h_w(\til a_{wi,j})$.
By choice of $\theta_{wi,j}$, $g_{wi,j}^L$ does not leave the ball $2B_{wij}$, hence does not exit $\cal N_{wi}$.
The superscript ``$L$'' here is to denote that this geodesic and $t^L_{wi,j}$ are defined on the left side of the interval $I_{wi,j}$.
A computation similar to \eqref{eq:similarity} gives the same bounds for the scaling factor of $g_{wi,j}^L$.

\begin{figure}[ht]
\centering
\includegraphics[width=0.8\textwidth]{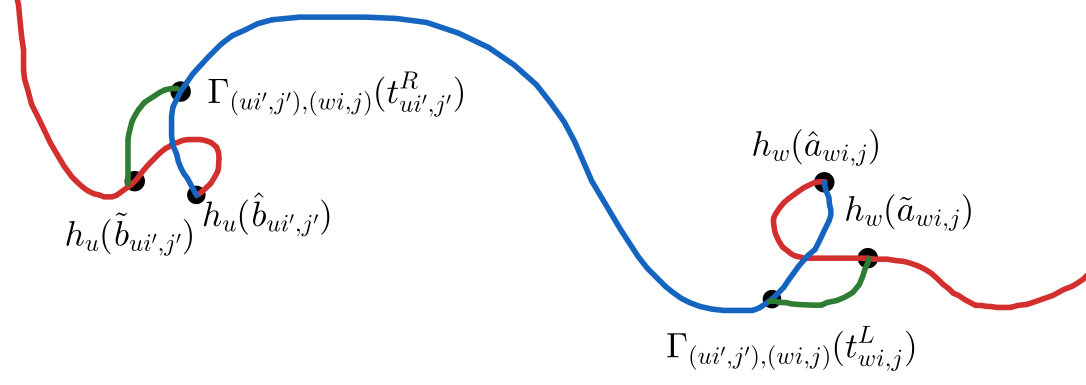}
\caption{In this figure, the red curves are the images of $h_u|\hat{I}_{ui',j'}$ (left) and $h_w|\hat{I}_{wi,j}$ (right), the blue curve is the curve $\Gamma_{(ui',j'),(wi,j)}$ and the green curves are the geodesics $g^R_{ui',j'}$ (left) and $g^L_{wi,j}$ (right).}
\label{fig:modif}
\end{figure}

\subsection{Local modifications on intervals of \texorpdfstring{$\mathscr{I}_2$}{I2}}

Let $I_{wi,j} \in \mathscr{I}_2 \setminus \{I_{1,0}\}$. Again we assume that $I_{ui',j'}$ is the interval to the left of $I_{wi,j}$.
From the prior subsection, we have already defined points $t_{ui',j'}^L$, $\til a_{ui',j'}$, $\til b_{wi,j}$, and $t_{wi,j}^R$, along with geodesics $g_{ui',j'}^L:[t_{ui',j'}^L,\til a_{ui',j'}] \to X$ and $g_{wi,j}^R:[\til b_{wi,j},t_{wi,j}^R] \to X$.
We proceed to define an arc $\Gamma_{(ui',j'),(wi,j)}$, points $\til b_{ui',j'}$, $t_{ui',j'}^R$, $t_{wi,j}^L$, and $\til a_{wi,j}$, and geodesics $g_{ui',j'}^R:[\til b_{ui',j'},t_{ui',j'}^R] \to X$ and $g_{wi,j}^L:[t_{wi,j}^L,\til a_{wi,j}] \to X$ as in the last subsection.
The only difference is that we take into account the modifications performed in \textsection\ref{sec:modif1}.
Set
\begin{align*}
\til b_{ui',j'} &= \min\{t \in [\til a_{ui',j'},\hat b_{ui',j'}] : \dist(h_u(t),\Gamma_{(ui',j'),(wi,j)}([\hat b_{ui',j'},\hat a_{wi,j}])) = \theta_{wi,j}\diam U_w\}, \\
t_{ui',j'}^R &= \max\{t \in [\hat b_{ui',j'},\hat a_{wi,j}] : d(h_u(\til b_{ui',j'}),\Gamma_{(ui',j'),(wi,j)}(t)) = \theta_{wi,j}\diam U_w\}, \\
\til a_{wi,j} &= \max\{t \in [\hat a_{wi,j},\til b_{wi,j}] : \dist(h_w(t),\Gamma_{(ui',j'),(wi,j)}([\hat b_{ui',j'},\hat a_{wi,j}])) = \theta_{wi,j}\diam U_w\},\\
t_{wi,j}^L &= \min\{t \in [t_{ui',j'}^R,\hat a_{wi,j}] : d(\Gamma_{(ui',j'),(wi,j)}(t),h_w(\til a_{wi,j})) = \theta_{wi,j}\diam U_w\}.
\end{align*}
Inequalities analogous to \eqref{eq:bt}, \eqref{eq:at}, \eqref{eq:at2}, and \eqref{eq:tt} hold for these points as well.

Fix a geodesic $g_{ui',j'}^R:[\til b_{ui',j'},t_{ui',j'}^R] \to X$ joining $h_u(\til b_{ui',j'})$ to $\Gamma_{(ui',j'),(wi,j)}(t_{ui',j'}^R)$ and a geodesic $g_{wi,j}^L:[t_{wi,j}^L,\til a_{wi,j}] \to X$ joining $\Gamma_{(ui',j'),(wi,j)}(t_{wi,j}^L)$ to $h_w(\til a_{wi,j})$.
Similar to \textsection\ref{sec:modif1}, the choice of $\theta_{wi,j}$ implies $g_{ui',j'}^R$ and $g_{wi,j}$ do not leave $2B_{wij}$, thus do not exit $\cal N_{wi}$.
The same bounds as in \eqref{eq:similarity} hold for the scaling factors of $g_{ui',j'}^R$ and $g_{wi,j}^L$.

Lastly,
\begin{align*}
d(h_w(\til a_{wi,j}),&h_w(\til b_{wi,j}))\\ 
& \ge d(h_w(\hat a_{wi,j}),h_w(\hat b_{wi,j})) - d(h_w(\hat a_{wi,j}),h_w(\til b_{wi,j})) - d(h_w(\hat b_{wi,j}),h_w(\til a_{wi,j})) \\
& \ge \kappa_w\diam U_w - \theta_{wi,j}\diam U_w - \theta_{wi,j}\diam U_w \\
& \ge \theta_{wi,j}\diam U_w,
\end{align*}
hence $\til b_{wi,j} - \til a_{wi,j} >0$ and, if $E$ is $c$-uniformly disconnected,
\begin{equation}\label{eq:ba}
\til b_{wi,j} - \til a_{wi,j} \ge (L\Lambda_w)^{-1}\theta\diam U_w.
\end{equation}

\begin{figure}[ht]
\centering
\includegraphics[width=0.55\textwidth]{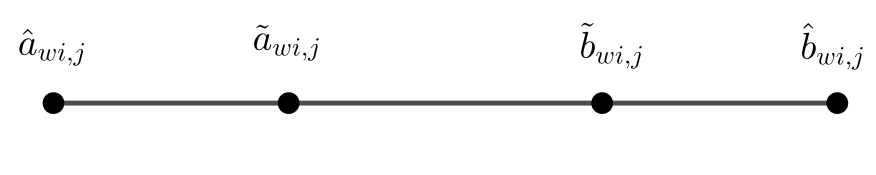}
\caption{The positions of $\hat{a}_{wi,j}$, $\til a_{wi,j}$, $\til b_{wi,j}$, and $\hat{b}_{wi,j}$.}
\end{figure}

Define $\gamma_{(ui',j'),(wi,j)}:[\til b_{ui',j'},\til a_{wi,j}] \to X$
by
\begin{equation} \label{eq:gamma}
\gamma_{(ui',j'),(wi,j)}(t) = \begin{cases} g_{ui',j'}^R(t), & t \in [\til b_{ui',j'},t_{ui',j'}^R], \\ \Gamma_{(ui',j'),(wi,j)}(t), & t \in [t_{ui',j'}^R,t_{wi,j}^L], \\ g_{wi,j}^L(t), & t \in [t_{wi,j}^L,\til a_{wi,j}], \end{cases}
\end{equation}
and note that $\gamma_{(ui',j'),(wi,j)}$ does not leave the ball $2B_{wij}$, since none of its components do.
By \eqref{eq:distXwiXwj}, note that if $\gamma_{(ui',j'),(wi,j)}$ and $\gamma_{(vi'',j''),(\omega\iota,\zeta)}$ are any two such curves, then
\begin{align}
\dist(\gamma_{(ui',j'),(wi,j)},\gamma_{(vi'',j''),(\omega\iota,\zeta)}) & \ge \dist(B_{wij},B_{vi''j''}) \label{eq:gammadist} \\
& \ge \tfrac{1}{12}(\rho_{wi}\epsilon_{wij}\diam U_{wij} + \rho_{vi''}\epsilon_{vi''j''}\diam U_{vi''j''}). \notag
\end{align}
Further, by \eqref{eq:beta*} and the choice of $\theta_{wi,j}$,
\begin{equation}\label{eq:gammaH}
\dist(z,H_{(ui',j'),(wi,j)}) \ge \frac{\beta^*}{L}\diam U_w - \theta_{wi,j}\diam U_w \ge \theta_{wi,j}\diam U_w
\end{equation}
for each $z \in \gamma_{(ui',j'),(wi,j)}$.

%In the following Lemma, we show that the concatenation of the arcs $\Im g_{wi,j}^R$, $\Im g_{wi,j}^L$, and $\Gamma_{(ui',j'),(wi,j)}([t_{ui',j'}^4,t_{wi,j}^1])$ is bi-Lipschitz, and, if $E$ is $c$-uniformly disconnected, is a quasisimilarity, quantitatively.

\section{Proof of Theorem \ref{thm:main1} and Theorem \ref{Thm:main2}}\label{sec:proof}

In this section we give the proof of Theorems \ref{thm:main1} and \ref{Thm:main2}. We start by defining a map 
\[ \Phi : [\hat{a}_{1,0},\hat{b}_{1,n_1}] \to U_1 .\] 
The definition of $\Phi$ is done in four steps. 

\begin{itemize}
\item Recall that $I_{1,0}$ is adjacent only to $I_{11,0}$, and that $I_{11,0}$ lies to the right of $I_{1,0}$. Moreover, a point $\til b_{1,0}$ has been defined but no point $\til a_{1,0}$. Define
\[ \Phi|[\hat{a}_{1,0},b_{1,0}](t) = 
\begin{cases}
h_{\varepsilon}(t), & t \in [\hat{a}_{1,0},\til b_{1,0}],\\
\gamma_{(1,0),(11,0)}(t), & t \in [\til b_{1,0}, b_{1,0}].
\end{cases} \]
\item Recall that $I_{1,n_1}$ is adjacent only to $I_{1n_1,n_{1n_1}}$, and that $I_{1n_1,n_{1n_1}}$ lies to the right of $I_{1,0}$. Moreover, a point $\til a_{1,n_1}$ has been defined but no point $\til b_{1,n_1}$. Define
\[ \Phi|[a_{1,n_1},\hat{b}_{1,n_1}](t) = 
\begin{cases}
\gamma_{(1n_1,n_{1n_1}),(1,n_1)}(t), & t \in [a_{1,n_1}, \til a_{1,n_1}],\\
h_{\varepsilon}(t), & t \in [\til a_{1,n_1}, \hat{b}_{1,n_1}].
\end{cases} \]
\item Let $I_{wi,j} \in \mathscr{I} \setminus \{I_{1,0},I_{1,n_1}\}$. Then there exists an interval $I_{ui',j'}$ adjacent to $I_{wi,j}$ on the left and there exists an interval $I_{vi'',j''}$ adjacent to $I_{wi,j}$ on the right. Define
\[ \Phi|I_{wi,j}(t) = 
\begin{cases}
\gamma_{(ui',j'),(wi,j)}(t), & t \in [a_{wi,j}, \til a_{wi,j}],\\
h_{w}(t), & t \in [\til a_{wi,j}, \til b_{wi,j}],\\
\gamma_{(wi,j),(vi'',j'')}(t), & t \in [\til b_{wi,j}, b_{wi,j}].
\end{cases} \]
\item It remains to define $\Phi$ on $J_{\infty}$. Define $\textbf{W}$ to be the set of all sequences $\textbf{a}\in\N^{\N}$ such that $\textbf{a}(m)\in \W$ for all $m\in\N$. (Here, if $\textbf{a} =i_1i_2\cdots\in\N^{\N}$, then $\textbf{a}(m) = i_1\cdots i_m\in\N^{m}$ is the truncation of $\textbf{a}$.) If $\textbf{w} \in \textbf{W}$, then by Lemma \ref{Lem:Decomposition}, $|J_{\textbf{w}(m)}| \leq 3^{1-|\textbf{w}(m)|}|J_{1}| = 3^{1-|\textbf{w}(m)|}$ and, by Proposition \ref{Prop:DefiningSequence}, $\diam{U_{\textbf{w}(m)}} \leq 2^{-|\textbf{w}(m)|}\diam{U_{\varepsilon}}$. Therefore, since $X$ is complete, elements of $J_{\infty}$ and elements of $E$ are uniquely determined by elements of $\textbf{W}$. In particular, there exist bijections $\pi_1: \textbf{W} \to J_{\infty}$ and $\pi_2 : \textbf{W} \to E$ such that for each $\textbf{w}\in \textbf{W}$, $\pi_1(\textbf{w})$ is the unique element of $\bigcap_{m=1}^{\infty}J_{\textbf{w}(m)}$ and $\pi_2(\textbf{w})$ is the unique element of $\bigcap_{m=1}^{\infty}U_{\textbf{w}(m)}$. Define now $\Phi$ on $[\hat{a}_{1,0},\hat{b}_{1,n_1}]$ by setting
\[
\Phi|J_{\infty} = \pi_2\circ (\pi_1)^{-1}.\]
\end{itemize}

The following proposition, which is the goal of this section, completes the proof of Theorem \ref{thm:main1} and Theorem \ref{Thm:main2}.

\begin{Prop}\label{Prop:Quasisymmetric}
The map $\Phi$ is an embedding of $[\hat{a}_{1,0},\hat{b}_{1,n_1}]$ into $X$ such that $\Phi(J_{\infty}) = E$ and $\Phi$ is locally bi-Lipschitz away from $J_{\infty}$. Furthermore, if $E$ is $c$-uniformly disconnected, there is a constant $H(\texttt{data}) \ge 1$ such that $\Phi$ is an $H$-weak quasisymmetry.
\end{Prop}

For the proof of the proposition, we require some lemmas.

\begin{Lem}\label{lem:inclusion}
For all $w\in \W \setminus\{\varepsilon, 1\}$ we have $\Phi(J_w) \subset \mathcal{N}_w$.
\end{Lem}

\begin{proof}
We have that 
\[ J_w  = (J_{\infty} \cap J_w) \cup \bigcup_{wu \in \W}\bigcup_{j=0}^{n_{wu}}I_{wu,j} \]
for all $w,wu \in \W$.
For a fixed $wu \in \W$ and $j\in\{0,\dots,n_{wu}\}$, there exists $v\in \W\setminus\{\varepsilon\}$ and $i\in\{1,\dots,n_{v}\}$ such that $wu = vi$. Note that $v$ is either $w^{\uparrow}$ or a descendant of $w^{\uparrow}$. There exists an interval $I_{v'i',j'}$ adjacent to $I_{vi,j}$ on the left, and there exists an interval $I_{v''i'',j''}$ adjacent to $I_{vi,j}$ on the right. Then,
\begin{equation}\label{eq:union}
\Phi(I_{vi,j}) \subset h_v(I_{vi,j}) \cup \gamma_{(v'i',j'),(vi,j)} \cup \gamma_{(vi,j),(v''i'',j'')}.
\end{equation}
By design, each of the sets in the right-hand side of \eqref{eq:union} is contained in $\mathcal{N}_{vi}$ (see Proposition \ref{Prop:AllGen}, the discussion after \eqref{eq:beta*}, the discussion after \eqref{eq:t3t4}, and the discussion after \eqref{eq:tt}). Since $\mathcal{N}_{vi} \subset \mathcal{N}_w$, we conclude that $\Phi(J_w\setminus J_{\infty}) \subset \mathcal{N}_w$. Therefore,
\[\Phi(J_w) = \Phi(J_{w} \setminus J_{\infty}) \cup \pi_2\circ (\pi_1)^{-1}(J_{w} \cap J_{\infty}) = \Phi(J_{w} \setminus J_{\infty}) \cup (E\cap U_{w}) \subset \mathcal{N}_{w}. \qedhere\]
\end{proof}

\begin{Lem}\label{Lem:QuasisimilarConcatenation}
Let $I_{wi,j} \in \mathscr I$.
Suppose that $I_{ui',j'}$ is adjacent to $I_{wi,j}$ on the left and $I_{vi'',j''}$ is adjacent to $I_{wi,j}$ on the right and that $\Phi$ is defined on $I_{ui',j'} \cup I_{wi,j} \cup I_{vi'',j''}$.
Then $\Phi|(I_{ui',j'} \cup I_{wi,j} \cup I_{vi'',j''})$ is a bi-Lipschitz embedding.
Furthermore, if $E$ is $c$-uniformly disconnected, then there is a constant $\til L(\texttt{data}) \ge 1$ so that $\Phi|(I_{ui',j'} \cup I_{wi,j} \cup I_{vi'',j''})$ is an $(\til L,\Lambda_w)$-quasisimilarity.
\end{Lem}

\begin{proof}
Let $I_{\upsilon \iota,\zeta}$ be the interval to the left of $I_{ui',j'}$ and let $I_{\omega \iota',\zeta'}$ be the interval to the right of $I_{vi'',j''}$.
For brevity, we write $\upsilon$ in place of $(\upsilon\iota,\zeta)$, $u$ in place of $(ui',j')$, $w$ in place of $(wi,j)$, $v$ in place of $(vi'',j'')$, and $\omega$ in place of $(\omega\iota',\zeta')$.
We make the assumption that both $t_u^L,t_v^R \in [a_u,b_v]$, but it does not affect the proof if this is not the case.
There are thirteen subintervals to be considered, which are given by adjacent endpoints of the following string of inequalities:
\[ a_u < t_u^L < \til a_u < \til b_u < t_u^R < t_w^L < \til a_w < \til b_w < t_w^R < t_v^L < \til a_v < \til b_v < t_v^R < b_v. \]
Fix $s,t \in I_u \cup I_w \cup I_v$.
By design, $\Phi$ restricted to each subinterval is bi-Lipschitz, and if $E$ is uniformly disconnected, then $\Phi$ restricted to each subinterval is a quasisimilarity, quantitatively, so we need not worry about the case in which $s$ and $t$ lie in the same subinterval.

If $s$ and $t$ lie in adjacent subintervals then one lies in a subinterval on which $\Phi$ is defined by a geodesic and the other lies in a subinterval on which $\Phi$ is defined by a bi-Lipschitz arc.
Say $\Phi(s)$ lies on a geodesic $g$ and $\Phi(t)$ lies on a bi-Lipschitz arc $f$.
Since this geodesic $g$ is chosen to join closest points of two bi-Lipschitz arcs (with $f$ being one of them), \cite[Lemma 4.2]{HVZ24} implies that the concatenation of $g$ and $f$ is still bi-Lipschitz.
(The idea of the cited lemma is that a geodesic joining some point $p$ to a nearest point of a bi-Lipschitz arc must intersect the bi-Lipschitz arc at a large angle, otherwise it is not a nearest point to $p$.)
If $E$ is uniformly disconnected, then the maps $g$ and $f$ are quasisimilarities with scaling factor $\Lambda_w$.
Rescaling the metric by this factor, these quasisimilarities become bi-Lipschitz arcs, where again \cite[Lemma 4.2]{HVZ24} can be applied (the notion of closest points is preserved under this scaling of the metric).
Upon returning to the original metric, the bi-Lipschitz concatenation becomes a quasisimilarity with scaling factor $\Lambda_w$.

Assume now that $s$ and $t$ are separated by a subinterval above.
%In the case when $E$ is not uniformly disconnected, the image points 
Then $\Phi(s)$ and $\Phi(t)$ lie in disjoint bi-Lipschitz arcs, hence we obtain the desired bounds and the first claim of the lemma. We now assume for the rest of the proof that $E$ is uniformly disconnected. Note that $s,t \in J_{w^{\uparrow}}$.
By Lemma \ref{lem:inclusion}, $\Phi(s),\Phi(t) \in \cal N_{w^{\uparrow}}$, so by \eqref{eq:diamUwi},
\[ d(\Phi(s),\Phi(t)) \le \diam\cal N_{w^{\uparrow}} \le 2\diam U_{w^{\uparrow}} \le 8\xi\diam U_w. \]
On the other hand, $s$ and $t$ are separated by a subinterval.
By \eqref{eq:t3t4}, \eqref{eq:at2}, \eqref{eq:tt}, and \eqref{eq:ba}, the diameter of the smallest subinterval above is still of length at least
\[ \theta\min\{L^{-1},(L^*)^{-1}\}\min\{\Lambda_u^{-1}\diam U_u,\Lambda_w^{-1}\diam U_w,\Lambda_v^{-1}\diam U_v\}. \]
By using \eqref{eq:Lambda}, we obtain
\[ |s - t| \ge (2N + 1)^{-1}\theta\min\{L^{-1},(L^*)^{-1}\}\Lambda_w^{-1}\diam U_w. \]
It remains to provide lower bounds on $d(\Phi(s),\Phi(t))$ in terms of $|s - t|$, which we do by considering three cases.
%when both $s$ and $t$ belong to distinct subintervals on which $\Phi$ is defined by $h_u$, $h_w$, or $h_v$; the case in which both $s$ and $t$ belong to subintervals on which $\Phi$ is defined by $\gamma_{\upsilon,u}$, $\gamma_{u,w}$, $\gamma_{w,v}$, or $\gamma_{v,\omega}$; and the case in which $s$ belongs to a subinterval on which $\Phi$ is defined by $h_u$, $h_w$, or $h_v$, and $t$ belongs to a subinterval on which $\Phi$ is defined by $\gamma_{\upsilon,u}$, $\gamma_{u,w}$, $\gamma_{w,v}$, or $\gamma_{v,\omega}$.

\emph{Case 1.} Suppose that $s$ and $t$ lie in distinct subintervals on which $\Phi$ is defined by $h_u$, $h_w$, or $h_v$.
Namely, $s,t \in [\til a_u,\til b_u] \cup [\til a_w,\til b_w] \cup [\til a_v,\til b_v]$.
For convenience, we assume that $s \in [\til a_u,\til b_u]$ and $t \in [\til a_w,\til b_w]$, but the other cases are identical.
By Proposition \ref{Prop:AllGen}(4), we have
\[ d(\Phi(s),\Phi(t)) \ge \dist(\Im(h_u),\Im(h_w)) \ge L^{-1}\diam U_w, \]
while, since $s,t \in J_{w^{\uparrow}}$, by \eqref{eq:Lambda},
\begin{equation}\label{eq:stdiamUw}
|s - t| \le |J_{w^{\uparrow}}| \le (2n_{w^{\uparrow}} + 1)(2n_w + 1)|J_{wi}| \le (2N + 1)^2|J_{wi}| = (2N + 1)^2\Lambda_w^{-1}\diam U_w.
\end{equation}

\emph{Case 2.} Suppose that $s$ and $t$ lie in distinct subintervals on which $\Phi$ is defined by $\gamma_{\upsilon,u}$, $\gamma_{u,w}$, $\gamma_{w,v}$, or $\gamma_{v,\omega}$.
Namely, $s,t \in [a_u,\til a_u] \cup [\til b_u,\til a_w] \cup [\til b_w,\til a_v] \cup [\til b_v,b_v]$.
For convenience, we assume that $s \in [\til b_u,\til a_w]$ and $t \in [\til b_w,\til a_v]$, but the other cases are identical.
By \eqref{eq:gammadist} and \eqref{eq:diamUwi},
\[ d(\Phi(s),\Phi(t)) \ge \dist(\gamma_{u,w},\gamma_{w,v}) \ge \tfrac{1}{12}\rho_w\epsilon_{wi}\diam U_{wi} \ge \tfrac{1}{48}\xi^{-1}\rho_w\epsilon_{wi}\diam U_w, \]
while the bound of \eqref{eq:stdiamUw} still holds.
Recall that since $E$ is uniformly disconnected, the constants $\rho_w$ and $\epsilon_{wi}$ are independent of $w$ and $wi$, respectively, and depend only on the $\texttt{data}$.

\emph{Case 3.} Suppose that $s$ lies in a subinterval on which $\Phi$ is defined by $h_u$, $h_w$, or $h_v$ and $t$ lies in a subinterval on which $\Phi$ is defined by $\gamma_{\upsilon,u}$, $\gamma_{u,w}$, $\gamma_{w,v}$, or $\gamma_{v,\omega}$.
Namely, $s$ belongs to an interval considered in Case 1, and $t$ belongs to an interval considered in Case 2.
For convenience, we assume that $s \in [\til a_u,\til b_u]$ and $t \in [\til b_w,\til a_v]$, but the other cases are identical.
By \eqref{eq:gammaH},
\[ d(\Phi(s),\Phi(t)) \ge \dist(\gamma_{w,v},\Im(h_v)) \ge \dist(\gamma_{w,v},H_{w,v}) \ge \theta\diam U_w \]
and again the bound of \eqref{eq:stdiamUw} holds.
\end{proof}

\begin{Lem}\label{Lem:PhiDist}
Suppose that $E$ is $c$-uniformly disconnected and that $I_{ui',j'}$ and $I_{wi,j}$ are not adjacent intervals.
Then
\[ \dist(\Phi(I_{ui',j'}),\Phi(I_{wi,j})) \ge (4\xi)^{-2}\theta\min\{\diam U_u,\diam U_w\}. \]
\end{Lem}

\begin{proof}
Let $s \in I_{ui',j'}$ and $t \in I_{wi,j}$.
There are three cases to consider.

\emph{Case 1.} Assume $s \in [\til a_{ui',j'},\til b_{ui',j'}]$ and $t \in [\til a_{wi,j},\til b_{wi,j}]$.
As $\Phi(s) = h_u(s)$ and $\Phi(t) = h_w(t)$, Proposition \ref{Prop:AllGen}(4) implies $d(\Phi(s),\Phi(t)) \ge L^{-1}\max\{\diam U_u,\diam U_w\}$.

\emph{Case 2.} Assume $s \in I_{ui',j'} \setminus [\til a_{ui',j'},\til b_{ui',j'}]$ and $t \in [\til a_{wi,j},\til b_{wi,j}]$.
Note that $\Phi(t) = h_w(t)$ and $\Phi(s) = \gamma(s)$ for some curve $\gamma$ defined as in \eqref{eq:gamma}.
In particular, either $\Phi(s) = \gamma_{(vi'',j''),(ui',j')}(s)$ or $\Phi(s) = \gamma_{(ui',j'),(vi'',j'')}(s)$ for some interval $I_{vi'',j''}$ adjacent to $I_{ui',j'}$.
In either case, by \eqref{eq:gammaH},
\[ d(\Phi(s),\Phi(t)) \ge \theta\diam U_u. \]

\emph{Case 3.} Assume $s \in I_{ui',j'} \setminus [\til a_{ui',j'},\til b_{ui',j'}]$ and $t \in I_{wi,j} \setminus [\til a_{wi,j},\til b_{wi,j}]$.
Note that either $\Phi(s) = \gamma_{(vi'',j''),(ui',j')}(s)$ or $\Phi(s) = \gamma_{(ui',j'),(vi'',j'')}(s)$ and either $\Phi(t) = \gamma_{(\omega i''',j'''),(wi,j)}(t)$ or $\Phi(t) = \gamma_{(wi,j),(\omega i''',j''')}(t)$.
By \eqref{eq:gammadist} and \eqref{eq:diamUwi}, we get
\begin{align*}
d(\Phi(s),\Phi(t)) & \ge \frac{1}{12}(\rho_{ui'}\epsilon_{ui'j'}\diam U_{ui'j'} + \rho_{wi}\epsilon_{wij}\diam U_{wij}) \\
& \ge (4\xi)^{-2}\theta\min\{\diam U_u,\diam U_w\}. \qedhere
\end{align*}
\end{proof}

\begin{proof}[{Proof of Proposition \ref{Prop:Quasisymmetric}}]
For the first claim, it suffices to prove that $\Phi$ is an embedding, as from Lemma \ref{Lem:QuasisimilarConcatenation} we have that $\Phi$ is locally bi-Lipschitz away from $J_{\infty}$. To prove continuity of $\Phi$ on $J_{\infty}$, fix $\bw \in \bW$ and fix $\e>0$. There exists $n\in\N$ such that $\diam{\mathcal{N}_{\bw(n)}} \leq 2^{-n}\diam{U_{\varepsilon}} < \e$. Then, the interior of $J_{\bw(n)}$ is a neighborhood of $\pi_1(\bw)$ with a diameter of at most $3^{-n}$ and $\Phi(\pi_1(\bw)) \in \Phi(J_{\bw(n)}) \subset \mathcal{N}_{\bw(n)}$ by Lemma \ref{lem:inclusion}. Therefore, $\Phi$ is continuous.

By design, it is easy to see that $\Phi|[\hat{a}_{1,0},\hat{b}_{1,n_1}]\setminus J_{\infty}$ is injective. Moreover, if $s\in J_{\infty}$ and $t\in [\hat{a}_{1,0},\hat{b}_{1,n_1}] \setminus J_{\infty}$, then $\Phi(s) \in E$ while $\Phi(t) \in X\setminus E$. Finally, since $\pi_1$ and $\pi_2$ are bijections, it follows that $\Phi|J_{\infty}$ is bijective. This shows that $\Phi$ is injective which completes the first claim of the proposition.

For the second claim, assume for the rest of the proof that $E$ is $c$-uniformly disconnected.
Let $s,t,\tau \in [\hat a_{1,0},\hat b_{1,n_1}]$ with $|s - t| \le |s - \tau|$.
The proof is a case study of the positions of $s$, $t$, and $\t$.

\emph{Case 1.} Suppose that $s$ and $\t$ are separated by a set $I_{wi,j}$ for some $w \in \W$, $i \in \{1,\dots,n_w\}$, and $j \in \{0,\ldots,n_{wi}\}$.
We may assume that $I_{wi,j}$ is a largest such interval, that is, $|w|$ is minimal among all words $v \in \W$ for which $I_{vi',j'}$ separates $s$ and $\t$ for some $i' \in \{1,\ldots,n_v\}$ and $j'
\in \{0,\dots,n_{vi'}\}$.

Recall that the interval $I_{wi,j}$ is contained in the interval $J_{wi}$, which in turn is contained in the interval $J_w$.
In addition to the subinterval $J_{wi}$, the interval $J_w$ contains two other subintervals of note: the interval $I_{w,i - 1}$ is adjacent to $J_{wi}$ on the left and the interval $I_{w,i}$ is adjacent to $J_{wi}$ on the right.
The interval $J_{wi}$ contains subintervals of the form $J_{wik}$ for $k \in \{1,\ldots,n_{wi}\}$ and of the form $I_{wi,k}$ for $k \in \{0,\ldots,n_{wi}\}$.
In particular, the interval $J_{wij}$ is adjacent to $I_{wi,j}$ on the left (except if $j = 0$, where instead $I_{w,i - 1}$ is adjacent to $I_{wi,j}$ on the left) and the interval $J_{wi(j + 1)}$ is adjacent to $I_{wi,j}$ on the right (except if $j = n_{wi}$, where instead $I_{w,i}$ is adjacent to $I_{wi,j}$ on the right).
Within the interval $J_{wij}$, there are subintervals of the form $J_{wijk}$ for $k \in \{1,\ldots,n_{wij}\}$ and of the form $I_{wij,k}$ for $k \in \{0,\ldots,n_{wij}\}$, while in the interval $J_{wi(j + 1)}$, there are subintervals of the form $J_{wi(j + 1)k}$ for $k \in \{1,\ldots,n_{wij}\}$ and of the form $I_{wi(j + 1)k}$ for $k \in \{0,\ldots,n_{wij}\}$.

This is the maximal extent to which we subdivide the intervals around $I_{wi,j}$: the furthest $s$ and $\t$ can be is when one is in $I_{w,i - 1}$ and the other is in $I_{w,i}$, as otherwise one of these intervals would separate $s$ and $\t$, which contradicts the maximality of $I_{wi,j}$.
We consider intervals up to one generation after $I_{wi,j}$ since we can obtain the desired bounds in these cases, but this is not always possible when considering intervals of the same generation of $I_{wi,j}$.

We note that since $s,\t \in J_w$, $|s - t| \le |s - \t| \le |J_w|$ and because $s \in J_w$, it must be that $s,t \in J_{w^{\uparrow}}$.
We now obtain the upper bound
\[ d(\Phi(s),\Phi(t)) \le \diam\cal N_{w^{\uparrow}} \le 2\diam U_{w^{\uparrow}} \le 8\xi\diam U_w \] 
by Lemma \ref{lem:inclusion} and \eqref{eq:diamUwi}.
The remainder of Case 1 is a calculation of a lower bound of $d(\Phi(s),\Phi(\t))$ in terms of $\diam U_w$ based on the placement of $s$ and $\t$ within the intervals specified above.

\emph{Case 1.1.} Suppose that $s$ and $\t$ lie in distinct intervals of the collection
\[ \{I_{w,i - 1},I_{w,i}\} \cup \{I_{wij,k} : k \in \{0,\ldots,n_{wij}\}\} \cup \{I_{wi(j + 1),k} : k \in \{0,\ldots,n_{wi(j + 1)}\}\}. \]
By Lemma \ref{Lem:PhiDist} and \eqref{eq:diamUwi},
\[ d(\Phi(s),\Phi(\t)) \ge (4\xi)^{-2}\theta\diam U_{wi} \ge (4\xi)^{-3}\theta\diam U_w. \]

\emph{Case 1.2.} Suppose that $s$ and $\t$ lie in distinct intervals of the collection
\[ \{J_{wijk} : k \in \{1,\ldots,n_{wij}\}\} \cup \{J_{wi(j + 1)k} : k \in \{1,\ldots,n_{wi(j + 1)}\}\}. \]
Note that since $s$ and $\t$ are separated by the interval $I_{wi,j}$, it must be the case that one lies in an interval of the left collection and the other must lie in the interval of the right collection, say $s \in J_{wijk_1}$ and $\t \in J_{wijk_2}$.
By Lemma \ref{lem:inclusion}, Lemma \ref{lem:distofhats}, and \eqref{eq:diamUwi},
\begin{align*}
d(\Phi(s),\Phi(\t)) & \ge \dist(\cal N_{wijk_1},\cal N_{wi(j + 1)k_2}) \ge \frac{(\diam U_{wijk_1} + \diam U_{wi(j + 1)k_2})}{(4\xi)^{3}}\ge \frac{\diam U_w}{2(4\xi)^{6}}.
\end{align*}

\emph{Case 1.3.} Suppose that $s$ lies in an interval in the collection considered in Case 1.1 and $\t$ lies in an interval in the collection considered in Case 1.2.
Again we note that one of $s$ or $\t$ lies to the left of $I_{wi,j}$ and the other lies to the right of $I_{wi,j}$.
As in Case 1.2, by Lemma \ref{lem:inclusion}, Lemma \ref{lem:distofhats}, and \eqref{eq:diamUwi},
\[ d(\Phi(s),\Phi(\t)) \ge 2(4\xi)^{-3}\diam U_w. \]

\emph{Case 2.} Suppose that $s$ and $\t$ lie in adjacent intervals, say $s \in I_{wi,j}$ and $\t \in I_{ui',j'}$.
The proof again divides into two subcases, this time based on the size of $|s - \t|$ in comparison to $\diam U_w$.

\emph{Case 2.1.} Suppose that $|s - \t| \le (2N + 1)^{-2}\Lambda_w^{-1}\diam U_w$.
Note that by \eqref{eq:Lambda},
\[ |s - t| \le |s - \t| \le (2N + 1)^{-2}\Lambda_w^{-1}\diam U_w \le (2n_w + 1)^{-1}(2n_{wi} + 1)^{-1}|J_{wi}| \le |I_{vi'',j''}|, \] 
where $I_{vi'',j''}$ is the other interval adjacent to $I_{wi,j}$.
It follows that $t \in I_{ui',j'} \cup I_{wi,j} \cup I_{vi'',j''}$, and by Lemma \ref{Lem:QuasisimilarConcatenation}, the map $\Phi|(I_{ui',j'} \cup I_{wi,j} \cup I_{vi'',j''})$ is a $(\til L,\Lambda_w)$-quasisimilarity.
Therefore
\[ d(\Phi(s),\Phi(t)) \le \til L\Lambda_w|s - t| \le \til L\Lambda_w|s - \t| \]
and $d(\Phi(s),\Phi(\t)) \ge (\til L)^{-1}\Lambda_w|s - \t|$.

\emph{Case 2.2.} Suppose that $|s - \t| > (2N + 1)^{-2}\Lambda_w^{-1}\diam U_w$.
As in Case 1, it follows that $s,t,\t \in J_{w^{\uparrow}}$.
By Lemma \ref{lem:inclusion} and \eqref{eq:diamUwi},
\[ d(\Phi(s),\Phi(t)) \le \diam\cal N_{w^{\uparrow}} \le 2\diam U_{w^{\uparrow}} \le 8\xi\diam U_w < 8(2N + 1)^2\Lambda_w|s - \t|. \]
On the other hand, by Lemma \ref{Lem:QuasisimilarConcatenation}, $d(\Phi(s),\Phi(\t)) \ge (\til L)^{-1}\Lambda_w|s - \t|$.
\end{proof}

\bibliographystyle{alpha}
\bibliography{biblio}
\end{document}